\documentclass[11pt, reqno]{amsart} 

\usepackage[utf8]{inputenc} 
\usepackage[T1]{fontenc} 
\usepackage{xcolor,colortbl}
\usepackage[backend=bibtex,style=alphabetic,maxnames=5,minnames=1, uniquename=false,uniquelist=false]{biblatex} 
\usepackage{amsmath}
\usepackage{amssymb}
\usepackage{amsfonts}
\usepackage{amsthm}
\usepackage{mathtools}
\usepackage{mathrsfs}
\usepackage{caption}
\usepackage{subcaption}

\usepackage{shuffle}

\usepackage{nicefrac}
\usepackage{hyperref}
\usepackage{breqn}

\usepackage{pgf,tikz}
\usetikzlibrary{matrix,arrows,calc}
\usepackage{tikz-cd}
\usetikzlibrary{cd}
\tikzcdset{scale cd/.style={every label/.append style={scale=#1},
		cells={nodes={scale=#1}}}}

\usepackage{comment}

\DeclareMathOperator{\et}{\acute{e}t}

\DeclareMathOperator{\dR}{dR}
\DeclareMathOperator{\ab}{\text{ab}}

\DeclareMathOperator{\Sel}{Sel}
\DeclareMathOperator{\Spec}{Spec}
\DeclareMathOperator{\Li}{Li}
\DeclareMathOperator{\li}{li}
\DeclareMathOperator{\loc}{loc}

\renewcommand{\AA}{\mathbb{A}}

\newcommand{\FF}{\mathbb{F}}

\newcommand{\PP}{\mathbb{P}}
\newcommand{\QQ}{\mathbb{Q}}
\newcommand{\RR}{\mathbb{R}}
\newcommand{\ZZ}{\mathbb{Z}}

\newcommand{\Qbar}{\overline{\QQ}}
\newcommand{\diff}{\mathop{}\!\mathrm{d}}

\newcommand{\rH}{\mathrm{H}}
\newcommand{\dX}{\mathcal{X}}

\newcommand{\Lie}{\operatorname{Lie}}

\def\multiset#1#2{\ensuremath{\left(\kern-.3em\left(\genfrac{}{}{0pt}{}{#1}{#2}\right)\kern-.3em\right)}}

\usepackage[colorinlistoftodos]{todonotes}
\definecolor{terracotta}{rgb}{0.89, 0.45, 0.36}
\definecolor{turquoise}{rgb}{0.19, 0.84, 0.78}

\newcommand{\xrev}[1]{\todo[color=cyan!70]{X: #1}}

\theoremstyle{plain}
\newtheorem{thm}{Theorem}
\newtheorem{cor}[thm]{Corollary}

\newtheorem{prop}[thm]{Proposition}
\newtheorem{theorem}[thm]{Theorem}
\newtheorem*{theorem*}{Theorem}
\newtheorem{theorema}{Theorem}

\newtheorem{lemma}[thm]{Lemma}

\newtheorem*{props}{Proposition}

\theoremstyle{definition}
\newtheorem{defn}[thm]{Definition}

\theoremstyle{remark}
\newtheorem{rem}[thm]{Remark}
\newtheorem*{rem*}{Remark}

\numberwithin{equation}{section}
\numberwithin{thm}{section}

\numberwithin{equation}{section}

\usepackage{enumitem}
\setenumerate[2]{label={(\alph*)},leftmargin=*,nolistsep,parsep=\parskip}
\setenumerate[1]{label={(\roman*)},leftmargin=*,nolistsep,parsep=\parskip}

\makeatletter
\let\@wraptoccontribs\wraptoccontribs
\makeatother

\renewcommand{\tfrac}[2]{\textstyle{\frac{#1}{#2}}}

\begin{document}
	
	\title[Refined Selmer Equations]{Refined Selmer Equations for the Thrice-Punctured Line in depth two}

	\author[]{Alex J. Best}
	\author[]{L. Alexander Betts}
	\author[]{Theresa Kumpitsch}
	\author[]{Martin L\"udtke}
	\author[]{Angus W. McAndrew}
	\author[]{Lie Qian}
	\author[]{Elie Studnia}
	\author[]{Yujie Xu}
	
	
	\address{Mathematics and Statistics Department, Boston University, Boston, MA 02215, USA}
	\email{\href{alexjbest@gmail.com}{alexjbest@gmail.com} }
	\email{\href{mailto:angusmca@bu.edu}{angusmca@bu.edu}}
	\address{Mathematics Department, Harvard University, Cambridge, MA 02138, USA}
	\email{\href{mailto:abetts@math.harvard.edu}{abetts@math.harvard.edu} }
	\email{\href{mailto:yujiex@math.harvard.edu}{yujiex@math.harvard.edu}}
	\address{Institut für Mathematik, Goethe–Universit\"at Frankfurt, Robert-Mayer-Stra{\ss}e 6–8, 60325}
	\email{\href{mailto:kumpitsch@math.uni-frankfurt.de}{kumpitsch@math.uni-frankfurt.de} }
	\address{Bernoulli Institute, Rijksuniversiteit Groningen, Nijenborgh 9, 9747 AG Groningen, The Netherlands}
	\email{\href{mailto:m.w.ludtke@rug.nl}{m.w.ludtke@rug.nl}}
	\address{Department of Mathematics
		Building 380, Stanford University, Stanford, CA 94305}
	\email{\href{mailto:
			lqian@stanford.edu}{lqian@stanford.edu} }
	\address{Ecole Normale Sup\'erieure, 75005 Paris}
	\email{\href{mailto:
			studnia@imj-prg.fr}{studnia@imj-prg.fr}}
	

	\begin{abstract}
		In \cite{kim2005motivic}, Kim gave a new proof of Siegel's Theorem that there are only finitely many $S$-integral points on $\PP^1_\ZZ\setminus\{0,1,\infty\}$. One advantage of Kim's method is that it in principle allows one to actually find these points, but the calculations grow vastly more complicated as the size of $S$ increases. In this paper, we implement a refinement of Kim's method to explicitly compute various examples where $S$ has size~$2$ which has been introduced in \cite{BD19}. In so doing, we exhibit new examples of a natural generalization of a conjecture of Kim.
	\end{abstract}

	\maketitle
	
	\setcounter{tocdepth}{1}
	\tableofcontents
	
	\section{Introduction}
	
	Let $S$ be a finite set of primes and let $\mathcal X= \PP^1_{\ZZ_S}\setminus \{0,1,\infty\}$ where $\ZZ_S$ denotes the ring of $S$-integers. By \textit{Siegel's Theorem}, the set $\mathcal{X}(\ZZ_S)$ is finite. One procedure to in principle compute the set $\mathcal{X}(\ZZ_S)$ was introduced by Kim \cite{kim2005motivic},
	who constructed a sequence of subsets, called the \textit{Chabauty--Kim loci}, for each prime $p\notin S$:
	\[\mathcal{X}(\ZZ_p) \supseteq \mathcal{X}(\ZZ_p)_{S,1} \supseteq \mathcal{X}(\ZZ_p)_{S,2} \supseteq \ldots \supseteq \mathcal{X}(\ZZ_S).  \]
	The number~$n$ is called the \emph{depth} of the Chabauty--Kim locus. Kim showed that the set $\mathcal X (\ZZ_p)_{S,n}$ is finite in sufficiently large depth $n\gg 0$, thus re-proving Siegel's Theorem. The main advantage of Kim's approach is that the sets~$\mathcal{X}(\ZZ_p)_{S,n}$ are in principle computable, giving one a theoretical algorithm for computing $\mathcal{X}(\ZZ_S)$. However, the calculations involved have proved difficult in practice, and all currently worked-out examples succeeded in computing some of the $\mathcal{X}(\ZZ_p)_{S,n}$ only when $|S|{}\le 1$ \cite{BDCKW,DCW,corwin-dan-cohen:computational_i}.
	
	In this paper we manage to push the boundaries of these computations to cover also the case~$|S|{}=2$ by using a refinement of Kim's method introduced by the second author and Netan Dogra. This modified method introduces \textit{refined Chabauty--Kim loci} 
	$\mathcal{X}(\ZZ_p)_{S,n}^{\min} \subseteq \mathcal{X}(\ZZ_p)_{S,n}$ which still contain the $S$-integral points.
	We compute these sets in the case that $n = 2$ and $|S|{}  \le 2$:
	
	\begin{theorema}[Proposition  \ref{prop:refinedCKset_card=1}]\label{thm:main_s=1}
		Let $S=\{2\}$ and let $p$ be an odd prime. Then $\mathcal X(\ZZ_p)^{\min}_{S,2}$ is equal to the set of nontrivial $(p-1)$-st roots of unity $\zeta \in \ZZ_p$ for which $\Li_2(\zeta) = 0$, along with all the images of this set under the natural action of~$S_3$ on $\mathcal X$. Here, $\Li_2$ is the $p$-adic dilogarithm.
	\end{theorema}
	
	\begin{theorema}[Theorem \ref{refined Chabauty-Kim equations}]\label{thm:main_s=2}
		Let $S = \{2, q\}$, and let $p \notin S$ be a prime. 
		Then~$\mathcal X(\ZZ_p)_{S,2}^{\min}$ is equal to the set of points $z\in\mathcal X(\ZZ_p)$ satisfying the equation
		\[ a_{2,q} \Li_2(z) = a_{q,2} \Li_2(1-z) \] 
		along with all of the images of this set under the natural action of~$S_3$ on $\mathcal X$. Here, $a_{2,q}$ and $a_{q,2}$ are certain computable $p$-adic numbers which are not both zero. (See \S\ref{subsec:locdepth2} for the definition of the constants $a_{\ell,q}$ and a description of an algorithm to compute them; this algorithm has been implemented in SageMath~\cite{dcw_coefficients}.)
	\end{theorema}
	
	
	
	These results should be understood in the context of Kim's Conjecture\footnote{Kim's method applies not just to the thrice-punctured line, but more generally to any $S$-integral model of a hyperbolic curve, and Kim's Conjecture is formulated for all such~$\mathcal X$.} \cite[Conjecture~3.1 \& \S8.1]{BDCKW}, which asserts that $\mathcal{X}(\ZZ_p)_{S,n}=\mathcal{X}(\ZZ_S)$ for~$n\gg0$. This has been verified in several small cases:
	\begin{itemize}
		\item when $S=\emptyset$, $n=2$, and $p < 10^5$ \cite[\S6]{BDCKW};
		\item when $S=\{2\}$, $n=4$, and $3\leq p\leq29$ \cite[\S8]{DCW16};
		\item when $S=\{3\}$, $n=4$, and $p\in\{5,7\}$ \cite[Theorem~1.3]{corwin-dan-cohen:computational_i}.
	\end{itemize}

	It is natural in this context to formulate a refined version of Kim's Conjecture: that
	\[
	\mathcal{X}(\ZZ_p)_{S,n}^{\min}=\mathcal{X}(\ZZ_S)
	\]
	for~$n\gg0$. This refined conjecture holds whenever Kim's original conjecture holds, and Theorems~\ref{thm:main_s=1} and~\ref{thm:main_s=2} provide new examples where we can verify our refined conjecture for the thrice-punctured line, namely:
	\begin{itemize}
		\item whenever $2 \not\in S$, for any~$n$ and~$p$ (see below);
		\item when $S = \{2\}$, $n=1$ and $p = 3$;
		\item when~$S=\{2\}$, $n=2$ and $3\leq p\leq 10^5$; and
		\item when~$S=\{2,q\}$, $n=2$ and~$p=3$, where~$q>3$ is either a Fermat or Mersenne prime, or is one of the primes
		\begin{align*}
			q =\; &19,\; 37,\; 53,\; 107,\; 109,\; 163,\; 181,\; 199,\; 269,\; 271,\; 379,\; \\
			&431,\; 433,\; 487,\; 523,\; 541,\; 577,\; 593,\; 631,\; 701,\; 739,\; \\
			&757,\; 809,\; 811,\; 829,\; 863,\; 883,\; 919,\; 937,\; 971,\; 991,\; \ldots
		\end{align*}
	\end{itemize}
	Notably, even in the case $S=\{2\}$, our refined conjecture holds in lower depth~$n$ than Kim's original conjecture, and use of refined Chabauty--Kim makes the case~$|S|{}=2$ also accessible in depth 2.
	
	
	Let us say a little more about the fourth of these points. Using Theorem~\ref{thm:main_s=2} and some Newton polygon analysis, we prove the following.
	
	
	\begin{props}[Proposition  \ref{newton polygon analysis result}] 
		Let $S = \{2, q\}$ for~$q>3$ prime, and let $p = 3$. Then the refined Chabauty--Kim set $\mathcal X(\ZZ_3)_{S,2}^{\min}$ contains $\{2,-1,\frac12\}$ and at most one more $S_3$-orbit of points. The second orbit is present if and only if 
		\begin{equation}\label{eq:valuation_criterion}\tag{$\dag$}
			\min\{ v_3(a_{2,q}), v_3(a_{q,2}) \} = 1 + v_3(\log(q)) \,.
		\end{equation}
	\end{props}
	
	
	We check by elementary means in Section \ref{sec: classification of solutions} that when $S = \{2,q\}$ and~$q>3$, the set $\mathcal{X}(\ZZ_S)$ consists of $\{2,-1,\frac12\}$ and at most one more $S_3$-orbit of points, which is present if and only if~$q$ is a Fermat or Mersenne prime. (If~$q$ is Fermat, then~$q\in\mathcal{X}(\ZZ_S)$; if~$q$ is Mersenne, then~$-q\in\mathcal{X}(\ZZ_S)$.) So in particular, we see that our refined version of Kim's conjecture holds for~$S=\{2,q\}$, $n=2$ and~$p=3$ whenever $q>3$ is either Fermat or Mersenne, and also whenever condition~\eqref{eq:valuation_criterion} fails. Using the code in \cite{dcw_coefficients}, the values of~$q < 1000$ for which condition~\eqref{eq:valuation_criterion} fails are exactly the 31 values of~$q$ listed earlier.
	
	If $S = \{2,3\}$ we cannot choose $p=3$ since $p$ must be not contained in~$S$; this case is thus not covered by the Proposition above. If $S = \{2,3\}$, the smallest possible choice for $p$ is $p=5$. We treat this case in Section~\ref{case3} and show that $\mathcal{X}(\ZZ_5)_{\{2,3\},2}^{\min}$ is strictly larger that $\mathcal{X}(\ZZ[1/6])$, i.e.\ the refined Kim's conjecture does not hold when $S = \{2,3\}$, $n=2$ and $p=5$.
	
	\begin{rem*}[{Remark~\ref{rem:empty_locus}}]
		If~$2\notin S$, then $\mathcal{X}(\ZZ_p)_{S,n}^{\min}$ is automatically empty, and hence the refined version of Kim's Conjecture for the thrice-punctured line always holds in this case, for any $n$ and $p$. Together with Theorems~\ref{thm:main_s=1} and~\ref{thm:main_s=2}, this gives an explicit description of $\mathcal{X}(\ZZ_p)_{S,2}^{\min}$ for any~$S$ of size~$\leq2$.
		
		
		In fact, Kim's Conjecture as originally formulated in \cite{BDCKW} also holds automatically for the thrice-punctured line whenever $S\not\ni2$. This observation subsumes the cases of Kim's Conjecture verified in \cite[\S6]{BDCKW} and \cite[Theorem~1.3]{corwin-dan-cohen:computational_i} above. What is proved in these papers is, in effect, a slightly stronger version of Kim's Conjecture, using the Selmer schemes as defined in \cite[p.~95]{kim-unipotent-albanese} in place of those defined in \cite[Definition~2.7 \& \S8.1]{BDCKW}.
	\end{rem*}
	
	Let us now say a few words about the refined Chabauty--Kim method. The usual Chabauty--Kim method, which in fact applies to a general hyperbolic curve $\mathcal X$, revolves around the study of two objects: the \emph{global Selmer scheme} $\Sel_{S,n}$ and the \emph{local Selmer scheme} $H^1_f(G_p,U^\et_n)$, both defined in terms of the $\QQ_p$-pro-unipotent \'etale fundamental group truncated in depth\footnote{that is, the quotient of this group by the $(n+1)$-st step in its lower central series}~$n$. These are both affine schemes of finite type over $\QQ_p$, and when the inequality
	\[
	\dim\Sel_{S,n} < \dim H^1_f(G_p,U^\et_n)
	\]
	holds, then the Chabauty--Kim locus $\mathcal X(\ZZ_p)_{S,n}$ is finite. Moreover, given a sufficiently explicit description of the local and global Selmer schemes, one can write down defining equations for the Chabauty--Kim locus $\mathcal X(\ZZ_p)_{S,n}$, in the form of Coleman analytic functions on $\mathcal X(\ZZ_p)$ which vanish on $\mathcal X(\ZZ_p)_{S,n}$.
	
	
	This theory was studied in detail in the case of $\mathcal X=\PP^1_{\ZZ_S}\setminus\{0,1,\infty\}$ and $n=2$ in work of Dan-Cohen and Wewers \cite{DCW}. There, they showed that $\dim H^1_f(G_p,U^\et_2)=3$, while $\dim\Sel_{S,2}=2|S|$. So the usual Chabauty--Kim method applies for this $\mathcal X$ whenever $|S|{}\leq1$, and in the case $|S|{}=1$, Dan-Cohen and Wewers found that the Chabauty--Kim locus $\mathcal X(\ZZ_p)_{S,2}$ is cut out by the equation
	\[
	2\Li_2(z)=\log(z)\log(1-z)
	\]
	(independent of $S$) \cite[\S12]{DCW}.
	
	The refined Chabauty--Kim method of \cite{BD19} replaces the global Selmer scheme by a \emph{refined global Selmer scheme} $\Sel_{S,n}^{\min}$ which is a closed subscheme of $\Sel_{S,n}$, and when the inequality
	\[
	\dim\Sel_{S,n}^{\min}<\dim H^1_f(G_p,U^\et_n)
	\]
	holds, then the refined Chabauty--Kim locus $\mathcal X(\ZZ_p)_{S,n}^{\min}$ is finite. In the particular case that $\mathcal X=\PP^1_{\ZZ_S}\setminus\{0,1,\infty\}$ and $n=2$, we have $\dim\Sel_{S,2}^{\min}={}|S|$, so the refined Chabauty--Kim method applies now whenever $|S|{}\leq2$. Moreover, using the explicit descriptions from \cite[\S12]{DCW}, 
	we can obtain explicit descriptions of the refined Chabauty--Kim loci, as in our Theorems~\ref{thm:main_s=1} and~\ref{thm:main_s=2}. Notably, refined Chabauty--Kim allows us to deal with the case $|S|{}=2$ already in depth $n=2$. And even in the case $|S|{}=1$, refined Chabauty--Kim provides more stringent constraints than usual Chabauty--Kim: the refined locus $\mathcal X(\ZZ_p)_{S,2}^{\min}$ for $S = \{2\}$ is
	the union of the $S_3$-translates of the set cut out by the two equations
	\[
	\log(z)=\Li_2(z)=0
	\]
	(the equation $\log(z)=0$ just says that~$z$ is a $(p-1)$st root of unity). The fact that we get two defining equations rather than one is significant in the context of Kim's conjecture: a generic pair of Coleman functions on a curve have no common zeroes, so heuristically one would expect any common zero of $\log(z)$ and $\Li_2(z)$ to be there for a reason. More specifically, it seems reasonable to conjecture that the only solution to $\log(z)=\Li_2(z)=0$ in $\mathcal X(\ZZ_p)$ is $z=-1$: this would imply the refined version of Kim's Conjecture for $S=\{2\}$, $n=2$ and the same prime $p$. For $p < 10^5$ congruent to $1$ mod $3$, \cite[\S6]{BDCKW} verified numerically the nonvanishing of $\Li_2(\zeta)$ for $\zeta$ a primitive 6-th root of unity; we extended this to all odd primes $p < 10^5$ and \emph{all} nontrivial $(p-1)$-st roots of unity~$\zeta$.
	
	We remark that in higher depth~$n$, the second, third and fourth authors have a proof of the refined Kim's Conjecture for $\mathcal{X}=\PP^1\setminus\{0,1,\infty\}$, $S=\{2\}$ and any odd~$p$, which will appear in forthcoming work. In fact truncating in depth $n \geq p-3$ is sufficient for arbitrary~$p\geq5$.
	
	\subsection*{Acknowledgements}
	This project was started during the 2020 Arizona Winter School as part of a project group guided by Minhyong Kim and the second author. The authors would like to thank Minhyong for his support and guidance throughout the different stages of this project, as well as the organizers of the Arizona Winter School for making this collaboration possible in the first place. We would also like to thank the referees for their very careful reading of this paper and their many helpful suggestions, which led to substantial improvements to the paper. The alternative approach to computing the coefficients $a_{\ell,q}$ in Remark~\ref{rem:differences} is due to them.
	
	\section{The $S$-Unit Equation and Classification of Solutions}
	\label{sec: classification of solutions}
	
	Before we begin the paper proper, let us recall a few elementary facts about $S$-integral points on the thrice-punctured line
	\[
	\mathcal X = \PP^1_{\ZZ_S}\setminus\{0,1,\infty\} = \Spec\left(\ZZ_S[u^{\pm1},v^{\pm1}]/(1-u-v)\right) \,.
	\]
	$S$-integral points on~$\mathcal X$ are the same thing as solutions to the \emph{$S$-unit equation}, i.e.\ they are elements $u\in\ZZ_S^\times$ such that $1-u\in\ZZ_S^\times$ also. Equivalently, $S$-integral points on~$\mathcal X$ correspond to solutions $(a,b,c)$ of the equation
	\begin{equation}
		\label{eq:unit}
		a + b = c 
	\end{equation}
	with $a,b,c \in \ZZ$ coprime and divisible only by primes in $S$ (up to identifying $(a,b,c)\sim(-a,-b,-c)$).
	
	The solutions to the $S$-unit equation can be determined when $|S|{}\leq2$ as follows.

	\begin{prop}
		\leavevmode
		\begin{itemize}
			\item If $S$ is a finite set of odd primes, then $\mathcal X(\ZZ_S)$ is empty.
			\item $\mathcal X(\ZZ_{\{2\}})$ consists of the $S_3$-orbit of the point $2$.
			\item If $q > 2$ is a prime, then $\mathcal X(\ZZ_{\{2,q\}})$ is exactly the union of $\mathcal X(\ZZ_{\{2\}})$ and the $S_3$-orbits of the following elements:
			\begin{itemize}
				\item[$\ast$] $q$ if $q$ is a Fermat prime;
				\item[$\ast$] $-q$ if $q$ is a Mersenne prime; and
				\item[$\ast$] $9$ if $q=3$. 
			\end{itemize}
			
		\end{itemize}
		In particular, for $S=\{2,3\}$, $\mathcal X(\ZZ_S)$ is equal to
		\begin{center} $\left\{2, \frac{1}{2}, -1\right\}\cup\left\{ 3,\frac{1}{3},\frac{2}{3}, \frac{3}{2}, -\frac{1}{2}, -2 \right\}\cup\left\{ 4, \frac{1}{4},  \frac{3}{4},  \frac{4}{3}, -\frac{1}{3},  -3 \right\}\cup\left\{9,  \frac{1}{9}, \frac{8}{9}, \frac{9}{8}, -\frac{1}{8},  -8\right\}.$\end{center}
		
	\end{prop}

	\begin{proof}
		Let $a+b=c$ be a solution to \eqref{eq:unit}: $a,b,c$ are coprime integers only divisible by primes in $S$ (where $S$ is a finite set of primes). Then $a,b,c$ are pairwise coprime. They cannot all be odd for parity reasons, so exactly one of them is even. In particular, for $(a,b,c)$ to exist, it is necessary that $2 \in S$.  
		
		Assume now that $S=\{2\}$. Among $a,b,c$, two of them are odd and their only prime divisor is $2$, so they are each $\pm 1$. Up to signed permutation\footnote{We call a \emph{signed permutation} an operation transforming a solution triple $(a,b,c)$ into another by permuting the components and adding signs as needed. In other words, the ``orbit'' of $(a,b,c)$ is $\{(a,b,c);(b,a,c);(b,-c,-a);(-c,b,-a);(a,-c,-b);(-c,a,-b)\}$. This is the form that the $S_3$-action takes in this notation.}, this means that $(a,b,c)=(2,-1,1)$. 
		
		Assume finally that $S=\{2,q\}$ where $q > 2$ is a prime. If two of $a,b,c$ are $\pm 1$, then up to signed permutation $(a,b,c)=(2,-1,1)$ again. If this is not the case, then one of $a,b,c$ is $\pm q^m$, one is $\pm 2^n$, and the last one is $\pm 1$, where $m,n \geq 1$ are integers. Up to signed permutation, we may then assume that $(a,b,c)=(q^m,-2^n, \pm 1)$. Now $m=1,c=1$ gives a solution if and only if $q$ is a Fermat prime; $m=1,c=-1$ gives a solution if and only if $q$ is a Mersenne prime. By the Catalan Conjecture (proved by Mih\u{a}ilescu in \cite{catalan-conj}), the only solution when $m \geq 2$ is $(9,-8,1)$ when $q=3$. (This special case of the Catalan Conjecture -- that the only perfect \emph{prime} powers differing by~$1$ are~$8$ and~$9$ -- can also be proved by elementary means.)
	\end{proof}

	\section{Refined Selmer Schemes and the $S_3$-Action}
	
	\subsection{The Chabauty--Kim method}
	To begin with, let us recall in outline the Chabauty--Kim method, as developed in \cite{kim2005motivic,kim-unipotent-albanese,BDCKW}.
	
	Let $S$ be a finite set of primes. Let $\ZZ_S$ denote the ring of $S$-integers, and let $\mathcal{X} \to \Spec(\ZZ_S)$ be a model of a hyperbolic curve over $\ZZ_S$ with generic fibre $X$. Assume for simplicity that~$\mathcal{X}$ has good reduction outside~$S$, i.e.\ is the complement of an \'etale divisor in a smooth proper curve over~$\ZZ_S$. We choose a place $p \not\in S$, a basepoint $b$, either $S$-integral or tangential (as introduced by Deligne in \cite[(15.9)]{deligne-droite-projective})
	, and denote by $U^{\et}$ and $U^{\dR}$ the $\QQ_p$-pro-unipotent \'etale fundamental group of~$(X_{\Qbar},b)$ and the pro-unipotent de Rham fundamental group of $(X_{\QQ_p},b)$, respectively. Let $U_n^{\et}$ and $U_n^{\dR}$ be the $n$th quotients along the lower central series.
	
	Following Kim, we consider the subspace
	\[H_f^1(G_p, U_n^{\et}) \subseteq H^1(G_p, U_n^{\et})\]
	consisting of $G_p$-equivariant right $U_n^\et$-torsors which are crystalline, where~$G_p$ denotes the absolute Galois group of~$\QQ_p$ (identified with a decomposition group in~$G_\QQ$). Crystalline $U_n^\et$-torsors are equivalent, via a Dieudonné functor, to admissible $U_n^{\dR}$-torsors, which are parametrized by the right coset space $F^0\backslash U_n^{\dR}$, see \cite[p.~119]{kim-unipotent-albanese} for the definition of this equivalence and \cite[Proposition~1.4]{kim-tangential-localization} for the proof.
	(Here, $F^0 = F^0 U_n^{\dR}$ refers to the Hodge subgroup of $U_n^{\dR}$.) The resulting isomorphism $H_f^1(G_p, U_n^{\et}) \cong F^0\backslash U_n^{\dR}$ is a non-abelian analogue of the Bloch--Kato logarithm.\footnote{Kim uses the left coset space $U_n^{\dR}/F^0$ rather than the right coset space $F^0\backslash U_n^{\dR}$. The two are equivalent via the inversion map. We prefer the latter, so that $H_f^1(G_p, U_n^{\et}) \cong F^0\backslash U_n^{\dR}$ specializes to the classical abelian Bloch--Kato logarithm for $n = 1$.}

Furthermore, we have the \textit{global $S$-Selmer scheme} of $\mathcal{X}$ in depth $n$, namely the subscheme 
\[\Sel_{S,n} = \Sel_{S,n}(\mathcal{X}) = H^1_{f,S}(G_{\QQ}, U_n^{\et}) \subseteq H^1(G_{\QQ}, U_n^{\et}),\]
consisting of $G_{\QQ}$-equivariant $U_n^{\et}$-torsors that are crystalline at $p$, and unramified at all places not equal to $p$ outside $S$ 
\cite[p.~120]{kim-unipotent-albanese}. 
\footnote{Strictly speaking, $H^1(G_{\QQ}, U_n^{\et})$ is not a scheme but only a functor. The subfunctor of torsors which are unramified outside $T \coloneqq S \cup \{p\}$, however, is representable by a $\QQ_p$-scheme of finite type. It agrees with $H^1(G_T, U_n^{\et})$ where $G_T$ is the largest quotient of $G_{\QQ}$ which is unramified outside~$T$. The Selmer scheme $H^1_{f,S}(G_{\QQ}, U_n^{\et})$ is a closed subscheme of $H^1(G_T, U_n^{\et})$.}
This gives rise to the following diagram, sometimes referred to as Kim's cutter,
\begin{equation}
	\begin{tikzcd}
		\label{kims cutter}
		\mathcal{X}(\mathbb Z_{S})\arrow[]{r}{} \arrow[]{d}[swap]{j_S} & \mathcal{X}({\mathbb Z_p}) {\arrow[]{d}{j_p}} \arrow[]{dr}{j_{\mathrm{dR}}} &  \\
		\Sel_{S,n} \arrow[]{r}{\mathrm{loc}_p} & H_f^1(G_p, U_n^{\et})  \arrow{r}{\sim} & F^0\backslash U_n^{\dR}
	\end{tikzcd}
\end{equation}
for all $n$.\footnote{Strictly speaking, the vertical arrows in the diagram don't make sense, since their domains are sets but their codomains are $\QQ_p$-schemes. These arrows in fact indicate maps from a set to the set of $\QQ_p$-points of the codomain, but this is customarily omitted from the notation.} Here the vertical arrows $j_S$, $j_p$, and $j_{\dR}$ denote the global, resp. local, resp. de Rham Kummer map, assigning to each $S$-integral (respectively, $p$-adic) point~$z$ the right torsor of paths from the fixed base point $b$ to~$z$
in the respective moduli space of torsors.\footnote{We use the functional convention for path composition, i.e. $\gamma_1 \gamma_2$ goes along $\gamma_2$ first and then along $\gamma_1$. Thus, the space of paths from ~$b$ to $z$ is a right torsor under the fundamental group at the base point $b$.}
The localization map~$\loc_p$ is the map on cohomology classes given by restriction along the natural map~$G_p\to G_{\QQ}$.

Using diagram~\ref{kims cutter}, we define the \emph{Chabauty--Kim locus} in depth~$n$
\[\mathcal{X}(\ZZ_p)_{S,n} \coloneqq j_p^{-1}(\loc_p(\Sel_{S,n}(\mathcal{X}))) \]
to be the preimage of the scheme-theoretic image of the localization map~$\loc_p$ under the local Kummer map~$j_p$. This is a subset of~$\mathcal{X}(\ZZ_p)$ containing $\mathcal{X}(\ZZ_S)$ by commutativity of~\eqref{kims cutter}, and the Chabauty--Kim loci form a nested sequence of subsets
\[\mathcal{X}(\ZZ_p) \supseteq \mathcal{X}(\ZZ_p)_{S,1} \supseteq \mathcal{X}(\ZZ_p)_{S,2} \supseteq \ldots \supseteq \mathcal{X}(\ZZ_S).  \]

A fundamental fact in the Chabauty--Kim method is that when the inequality
\begin{equation}\label{eq:c-k_ineq} \dim \Sel_{S,n} < \dim H_f^1(G_p, U_n^{\et}) \end{equation}
holds, then the set~$\mathcal{X}(\ZZ_S)$ of $S$-integral points is finite. More precisely, if we consider the ideal of algebraic functions vanishing on the scheme-theoretic image of the Selmer scheme, then this can be pulled back to a non-zero ideal of Coleman functions on $\mathcal{X}(\ZZ_p)$, and the zero locus of this pulled-back ideal is by definition $\mathcal{X}(\ZZ_p)_{S,n}$.

In particular, from a sufficiently explicit description of the bottom row of~\eqref{kims cutter}, one can give explicit equations for~$\mathcal{X}(\ZZ_p)_{S,n}$. Kim conjectured that the equality 
\[\mathcal{X}(\ZZ_p)_{S,n} = \mathcal{X}(\ZZ_S) \]
holds for large enough $n$ (cf. \cite[\S 1.4]{BDCKW}), and proposed this as a strategy for computing~$\mathcal{X}(\ZZ_S)$.

\subsection{Chabauty--Kim in depth $n \leq 2$ for $\mathcal{X} = \PP^1\setminus\{0,1,\infty\}$}
\label{subsec: explicit_c-k_depth_2}
Now let $\mathcal{X} = \mathbb{P}^1\setminus \{0,1,\infty\}$ be the thrice punctured line over $\ZZ$.
For a basepoint we take the tangential basepoint corresponding to $1$ in the tangent space at $0$ under the identification $T_0 \PP^1 = \AA^1$, denoted $\overset{\to}{01}$.
We briefly sketch what is known for the Chabauty--Kim method for $\mathcal{X}$ in depth $n \leq 2$, where all the maps in (\ref{kims cutter}) can be made explicit. We recall some basic facts, following \cite{kim2005motivic} and \cite{DCW}. 

\subsubsection{Depth 1}
\label{sec:basics_depth1}
Note that the geometric étale fundamental group $\pi_1(X_{\overline{\QQ}})$ of $X$ is the free profinite group in two generators corresponding to loops around $0$ and $1$. Hence, we have
\[
U_1^{\et} = (U^{\et})^{\ab} = (\pi_1(X_{\overline{\QQ}})^{\ab} \otimes \QQ_p) 
\cong \QQ_p(1) \oplus \QQ_p(1).
\]

Kummer theory yields an isomorphism
\[\Sel_{S,1} = H_{f,S}^1(G_{\QQ}, \QQ_p(1)^2) = H_{f,S}^1(G_{\QQ}, \QQ_p(1))^2 \cong \AA^S \times \AA^S,\]
where we choose coordinates $(x_{\ell})_{\ell \in S}$, $(y_{\ell})_{\ell \in S}$ of $\AA^S \times \AA^S$ in such a way that the global nonabelian Kummer map $j_S$ is given by
\[
z\mapsto ((v_{\ell}(z))_{\ell \in S}, (v_{\ell}(1-z))_{\ell \in S}).
\]
The de Rham side has an explicit description in depth 1 as well. Similar to the above, Kummer theory allows us to identify the local Selmer scheme with $\AA^2$ and $j_{\dR}$ is given by 
\[z \mapsto (\log(z), \log(1-z)) \]
in depth 1. Here $\log$ refers to the $p$-adic logarithm, defined as the Coleman integral
\begin{align*}
	\log(z) &= \int^z_{\overset{\to}{01}} \frac{\diff x}{x}.
\end{align*}
This gives a description of the localization map as
\begin{align*}
	\Sel_{S,1} = \AA^S \times \AA^S &\to H_f^1(G_p, \QQ_p(1)^2) = \AA^2\\
	((x_{\ell})_{\ell \in S},(y_{\ell})_{\ell \in S}) &\mapsto (\sum \log(\ell) x_{\ell}, \sum \log(\ell) y_{\ell}),
\end{align*}
and hence we have completely described the fundamental Chabauty--Kim diagram~\eqref{kims cutter} in depth 1: it looks like

\[
\begin{tikzcd}[column sep = small]
	z \in \hspace{-14em} & \mathcal{X}(\ZZ_S) \arrow[]{r}{}\arrow[]{dd}[swap]{j_S}\arrow[maps to, xshift=-2em, shorten=.3em]{ld}{} & \mathcal{X}(\ZZ_p) \arrow[]{dd}[swap]{j_p}\arrow[]{rdd}[]{j_\dR} & \\
	\scriptsize{\begin{matrix}
			((v_{\ell}(z))_{\ell \in S}, (v_{\ell}(1-z))_{\ell \in S}) \\
			\rotatebox[origin=c]{270}{$\in$}
	\end{matrix}} & & & \\[-25pt]
	\mathbb{A}^S\times\mathbb{A}^S \arrow[equal]{r} & \Sel_{S,1} \arrow[]{r}{\loc_p} & H^1_f(G_p,\QQ_p(1)^2) \arrow[equal]{r} & \mathbb{A}^2. \\[-25pt]
	& \scriptsize{\begin{matrix}
			\rotatebox[origin=c]{90}{$\in$} \\
			(x_\ell)_{\ell \in S}, (y_\ell)_{\ell \in S})
	\end{matrix}} \arrow[maps to, yshift=-5pt]{r}{} & \scriptsize{\begin{matrix}
			\rotatebox[origin=c]{90}{$\in$} \\
			(\sum\log(\ell)x_\ell, \sum\log(\ell)y_\ell)
	\end{matrix}} & 
\end{tikzcd}
\]

\subsubsection{Depth 2}
\label{sec:basics_depth2}

For $* = \et, \dR$ there is an exact sequence of algebraic groups over $\QQ_p$ 
\begin{equation}
	\label{eq:U2_sequence}
	\begin{tikzcd}
		1 \arrow[r] & (U^*)^{[2]}/(U^*)^{[3]} \arrow[d, "\cong" left] \arrow[r] & U_2^* \arrow[r] & U_1^* \arrow[d, "\cong" right] \arrow[r] & 1\\
		& \QQ_p(2) & & \QQ_p(1) \times \QQ_p(1), & 
	\end{tikzcd}
\end{equation}
where the Tate twist is to be interpreted in the respective realization.
The corresponding sequence on Lie algebras splits as Galois representations, see \cite[\S 5]{DCW}. 
The construction in \cite[\S5]{DCW} goes via the theory of motives, but can be described equivalently in terms of realizations.
The Lie algebra of $U_2^\et$ is unramified away from~$p$ and crystalline at~$p$, so its extension class lies in $\mathrm{Ext}^1(\QQ_p(1)^2,\QQ_p(2))=\rH^1_f(G_\QQ,\QQ_p(1))^2=0$.
Thus the extension of Lie algebras for $\Lie(U_2^\et)$ induced by \ref{eq:U2_sequence} has a unique $G_\QQ$-equivariant splitting, and applying the $D_\dR$ functor gives the corresponding splitting for $\Lie(U_2^\dR)$.

Hence, as in \cite[\S 5]{DCW}, via these splittings one can identify $U_2^*$ with the \textit{Heisenberg group}
\[
H^* = \begin{pmatrix}
	1 & \QQ_p(1) & \QQ_p(2)\\
	0 & 1 & \QQ_p(1) \\
	0 & 0 & 1
\end{pmatrix}.
\]

Note $H^i(G_{\QQ}, \QQ_p(2)) = 0$ for $i = 1,2$ by Soulé vanishing, see \cite[Theorem 1]{soule:higher} for the case $p$ odd.
The case $p = 2$ is addressed in unpublished work of Sharifi, see \cite{sharifi:soule}, but we will not need this.
Thus the abelianization map $\pi$ induces an isomorphism 
\[
\pi_* \colon \Sel_{S,2} \xrightarrow{\sim} \Sel_{S,1} = H_f^1(G_{\QQ}, \QQ_p(1)^2).
\]

Note that
\[
F^0 U_1^{\dR} = F^0 H^1_{\dR}(\mathcal{X})^\vee = 0,
\]
thus by \ref{eq:U2_sequence} we have $F^0 U_2^{\dR} = \{1\}$ and hence
\[
H_f^1(G_p, U_2^{\et}) \cong U_2^{\dR} = H^{\dR}.
\]

Finally, in the identification $U_2^{\dR} \cong \AA^3$, the map $j_{\dR}$ is given by locally $p$-adic analytic functions as
\[
z \mapsto (\log(z), \log(1-z), -\Li_2(z)).
\]

Note that $\Li_n$ denotes the $p$-adic polylogarithm, which is given as an iterated Coleman integral as
\[   
\Li_n(z) = \int^z_{\overset{\to}{01}} \underbrace{\frac{\diff x}{x}\cdots \frac{\diff x}{x}}_{n-1\text{ times}}\frac{\diff x}{1-x},
\]
where we follow Kim's convention (see page $109$ of \cite{kim-unipotent-albanese})
that the rightmost integrand is integrated ``first''.
They satisfy several useful identities (see \cite[Prop 6.4]{coleman82}):
\begin{align*}
	\Li_2(z) + \Li_2(1-z) &= -  \log(z) \log(1-z) ,\\
	\Li_2(z) + \Li_2(z^{-1}) &= -\frac 12 \log(z)^2 .
\end{align*}
We sum up what is known for the Chabauty--Kim method in depth 2 in the following diagram

\begin{equation}\label{kims cutter depth 2}
	\begin{tikzcd}[column sep = small]
		z \in \hspace{-14em} & \mathcal{X}(\ZZ_S) \arrow[]{r}{}\arrow[]{dd}[swap]{j_S}\arrow[maps to, xshift=-2em, shorten=.3em]{ld}{} & \mathcal{X}(\ZZ_p) \arrow[]{dd}[swap]{j_p}\arrow[]{rdd}[]{j_\dR} & \\
		\scriptsize{\begin{matrix}
				((v_{\ell}(z))_{\ell \in S}, (v_{\ell}(1-z))_{\ell \in S}) \\
				\rotatebox[origin=c]{270}{$\in$}
		\end{matrix}} & & & \\[-25pt]
		\mathbb{A}^S\times\mathbb{A}^S \arrow[equal]{r} & \Sel_{S,2} \arrow[]{r}{\loc_p}\arrow[equal]{d}[left]{\pi_*} & H^1_f(G_p,U_2^\et) \arrow[equal]{r}\arrow[]{d}{\pi_*} & \mathbb{A}^3 \arrow[]{d}{p_{1,2}} \\
		& \Sel_{S,1} \arrow[]{r}{\loc_p} & H^1_f(G_p,\QQ_p(1)^2) \arrow[equal]{r} & \mathbb{A}^2. \\[-25pt]
		& \scriptsize{\begin{matrix}
				\rotatebox[origin=c]{90}{$\in$} \\
				((x_\ell)_{\ell \in S}, (y_\ell)_{\ell \in S})
		\end{matrix}} \arrow[maps to, yshift=-5pt]{r}{} & \scriptsize{\begin{matrix}
				\rotatebox[origin=c]{90}{$\in$} \\
				(\sum\log(\ell)x_\ell, \sum\log(\ell)y_\ell)
		\end{matrix}} & 
	\end{tikzcd}
\end{equation}


			\subsection{The localization map in depth 2}\label{subsec:locdepth2}
			With the above choices of coordinates the localization map 
			\[h = \loc_p \colon \Sel_{S,2} = \AA^S \times \AA^S \to U_2^{\dR} = \AA^3 \]
			is of the form
			\[(x,y) = ((x_{\ell})_{\ell \in S}, (y_{\ell})_{\ell \in S}) \mapsto \left(\sum_{\ell \in S} \log(\ell) x_{\ell}, \sum_{\ell \in S} \log(\ell) y_{\ell}, h_3(x,y)\right), \]
			where the third component~$h_3$ is as yet undetermined. In \cite{DCW}, Dan-Cohen and Wewers study~$h_3$ using mixed Tate motives. They prove that it is bilinear,
			i.e.\ of the form
			\[h_3(x,y) = \sum_{\ell, q \in S} a_{\ell, q} x_{\ell} y_{q}, \]
			and give an algorithm based on Tate's computation of $K_2(\QQ)$ \cite[Theorem 11.6]{milnor} for computing the coefficients~$a_{\ell,q}\in\QQ_p$ in the case that~$\ell,q<p$ \cite[\S11]{DCW}.
			
			We want to explain here how to modify the algorithm of Dan-Cohen--Wewers to compute the coefficients~$a_{\ell,q}$ for all~$\ell,q\neq p$. From now until the end of this section, we assume that our prime~$p$ is odd. As in \cite{DCW}, the strategy revolves around two facts:
			
			\begin{enumerate}
				\item\label{condn:alq_normalisation} For $z \in \mathcal{X}(\ZZ_S)$ commutativity of~\eqref{kims cutter depth 2} yields
				\[h_3((v_{\ell}(z))_{\ell \in S}, (v_{\ell}(1 - z))_{\ell \in S}) = - \Li_2(z). \]
				\item\label{condn:alq_twisted_antisymmetry} The coefficients of this bilinear form satisfy a ``twisted antisymmetry relation'' (\cite[Prop. 10.4]{DCW})
				\[ a_{\ell,q} + a_{q,\ell} = \log(\ell)\cdot \log(q).\]
			\end{enumerate}
			
			\begin{rem}\label{rem:motivic_periods}
				The numbers $a_{\ell,q}$ are the $p$-adic realizations of the motivic periods $f_{\tau_{\ell} \tau_{q}}$ from \cite[§4.1]{corwin-dan-cohen:computational_i}, while $\log(\ell)$ is the $p$-adic realization of $f_{\tau_{\ell}}$. From this point of view, the twisted antisymmetry relation is a consequence of the shuffle product identity
				\[ f_{\tau_{\ell} \tau_q} + f_{\tau_q \tau_{\ell}} = f_{\tau_{\ell}} f_{\tau_q}. \]
			\end{rem}
			
			Note that each $S$-integral point on~$\mathcal{X}$ yields, via (i), a linear constraint on the coefficients~$a_{\ell,q}$. So, when there are sufficiently many $S$-integral points on~$\mathcal{X}$, these can determine the coefficients~$a_{\ell,q}$. For instance, we obtain particularly simple formulas for $a_{2,q}$ and $a_{q,2}$ if $q$ is a Fermat or Mersenne prime.
			
			\begin{lemma}
				\label{lem: Fermat Mersenne DCW coefficients}
				Let $q = 2^n \pm 1$ be a Fermat or Mersenne prime. Then the coefficients $a_{2,q}$ and $a_{q,2}$ are given by
				\[ a_{2,q} = -\frac1{n} \Li_2(1 \mp q), \quad a_{q,2} = -\frac1n \Li_2(\pm q). \]
			\end{lemma}
			
			\begin{proof}
				Let $S =\{2,q\}$. Then $1 \mp q = \mp 2^n$ is contained in $\mathcal X(\ZZ_S)$ since $1 - (1\mp q) = \pm q$ is also an $S$-unit. (Here, $\pm$ means~$+$ in the Fermat case and $-$ in the Mersenne case; $\mp$ denotes the opposite sign.) From~(i) we obtain
				\[ -\Li_2(1 \mp q) = h_3((n,0,0,1)) = a_{2,q} \cdot n \cdot 1, \]
				hence $a_{2,q} = -\frac1{n}\Li_2(1 \mp q)$. Using $\pm q \in \mathcal X(\ZZ_S)$ similarly gives the formula for $a_{q,2}$.
			\end{proof}
			
			In general, there may not be enough $S$-integral points on~$\mathcal{X}$ to determine the coefficients~$a_{\ell,q}$, so our strategy is to enlarge the set~$S$ so as to acquire enough $S$-integral points. In order to do this, the following lemma is essential.
			
			\begin{lemma}[{\cite[\S 10.2]{DCW}}]
				\label{coefficients invariant under enlarging S}
				Let $S \subseteq S'$ be finite sets of primes not containing $p$. Then the inclusion $\Sel_{S,2} \subseteq \Sel_{S',2}$ corresponds to the subspace inclusion $\AA^S \times \AA^S \subseteq \AA^{S'} \times \AA^{S'}$ with $x_{\ell'} = y_{\ell'} = 0$ for $\ell' \in S' \setminus S$, and the localization map $\loc_p$ on $\Sel_{S',2}$ restricts to the localization map on $\Sel_{S,2}$. In particular, the bilinear form coefficients $a_{\ell,q}$ of the third component of $\loc_p$ are independent of the set $S \supseteq \{\ell,q\}$ with $p \not\in S$.
			\end{lemma}

			As a result of Lemma~\ref{coefficients invariant under enlarging S}, the maps~$h_3$ for varying sets~$S$ induce a bilinear map
			\[
			h_3\colon E\otimes E \to \QQ_p
			\]
			where
			\[E=\QQ\otimes_{\ZZ}\ZZ_{(p)}^\times=\varinjlim_{S\not\ni p}\QQ^S,\] 
			which is an infinite-dimensional $\QQ$-vector space (written additively). For any element~$t\in\ZZ_{(p)}$ we write~$[t]:= 1\otimes t\in E$, so that $h_3([\ell],[q])=a_{\ell,q}$. Since $K_2( \ZZ_{(p)}) \otimes \QQ $ vanishes\footnote{
				This follows from the vanishing of $K_2(\QQ) \otimes \QQ$ and the existence of the short exact sequence
				\[0 \to K_2(\ZZ_{(p)}) \to K_2(\QQ) \to \bigoplus_{q \neq p} \kappa(q)^{\times} \to 0. \] 
				See \S 11 of \cite{milnor}.
			}, we know that
			the vector space $E\otimes E$ is spanned by vectors of the form $[t]\otimes [1-t]$
			with $t,1-t\in\ZZ_{(p)}^\times$, which we refer to as \textit{Steinberg elements}. 
			
			Our aim here is, given two prime numbers~$\ell$ and~$q$, to compute a decomposition of $[\ell] \otimes [q]$ as a sum of Steinberg elements and symmetric elements, i.e.\ generators of the form $[u] \otimes [v] + [v] \otimes [u]$ for $u,v \in \ZZ_{(p)}^\times$. Hence, we want a decomposition of the form
			\[
			[\ell]\otimes[q] = \sum_i\lambda_i\cdot([u_i]\otimes [v_i] + [v_i] \otimes [u_i]) + \sum_j\mu_j\cdot[t_j]\otimes[1-t_j]
			\]
			for some elements $u_i,v_i,t_j\in \ZZ_{(p)}^\times$ with $1-t_j\in\ZZ_{(p)}^\times$ and rational coefficients $\lambda_i,\mu_j$. Taking the image of such a decomposition under the map~$h_3$ yields a value for~$a_{\ell,q}$, namely
			\[
			a_{\ell,q} = \sum_i\lambda_i\log(u_i)\log(v_i) - \sum_j\mu_j\cdot\Li_2(t_j).
			\]
			
			To simplify the expressions, we do the computation in $\bigwedge^2E$, i.e.~first consider a decomposition
			\begin{equation}
				\label{eq: wedge decomposition}
				[\ell]\wedge[q] = \sum_i\lambda_i[t_i]\wedge[1-t_i].
			\end{equation}
			This then yields a decomposition in the tensor-square $E\otimes E$ of the desired form, namely
			\[
			[\ell]\otimes[q] = \frac12\bigl([\ell]\otimes[q]+[q]\otimes[\ell]\bigr) +  \frac12 \sum_i\lambda_i \Bigl([t_i]\otimes[1-t_i] - [1-t_i]\otimes[t_i]\Bigr),
			\]
			and so we obtain
			\begin{equation}
				\label{eq:coeff_h3}
				a_{\ell, q} = \frac12\log(\ell)\log(q) + \frac12\sum_i\lambda_i(\Li_2(1- t_i)-\Li_2(t_i)).
			\end{equation}
			
			\begin{rem}[Differences from \cite{DCW}]\label{rem:differences}
				Note that in \cite{DCW} the authors consider the vector space
				\[E' = \QQ \otimes \QQ^{\times} = \varinjlim_S \QQ^S, \]
				which is the $\QQ$-vector space spanned by all primes, including~$p$, and give an algorithm for a decomposition of $\ell \otimes q$ in $E' \otimes E'$. If $p > \ell, q$, the decomposition found by the algorithm in \cite{DCW} happens to only involve Steinberg elements $[t]\otimes[1-t]$ with $t,1-t\in\ZZ_{(p)}^\times$, so yields a description of the coefficient $a_{\ell,q}$ (in our notation). However, if~$q$ or~$\ell$ is larger than~$p$, then the decompositions produced by \cite{DCW} can involve Steinberg elements with $t\notin\ZZ_{(p)}^\times$, in which case the pairing~$h_3$ is not defined at~$[t]\otimes[1-t]$ and we cannot use the decomposition to control the value of~$a_{\ell,q}$. Our decomposition on the other hand allows us to specialize to an arbitrary odd prime $p$ as we avoid numbers containing factors of $p$.
				
				Nonetheless, it seems reasonable to expect that the pairing~$h_3$ should extend to a pairing on all of~$E'$ satisfying the conditions~\ref{condn:alq_normalisation} and~\ref{condn:alq_twisted_antisymmetry}, in which case the algorithm in \cite{DCW} would work without modification, as suggested to us by the referees. Indeed, as in Remark~\ref{rem:motivic_periods} the coordinates $a_{\ell,q}$ of the map~$h_3$ are in a natural way the $p$-adic realizations of motivic periods $f_{\tau_\ell\tau_q}$ which are unramified outside~$\{\ell,q\}$ and independent of the choices of~$p\notin\{\ell,q\}$ and~$S\supseteq\{\ell,q\}$. As explained in~\cite{chatzistamatiou-unver:p-adic_periods}, after choosing a branch of the $p$-adic polylogarithm one can extend the $p$-adic realization to motivic periods which are ramified at~$p$, and hence one can extend the pairing~$h_3$ to be defined on $E'\otimes E'$ by taking its coefficients to be the $p$-adic realization of $f_{\tau_\ell\tau_q}$, whether or not $p\in\{\ell,q\}$.
				
				
				What is missing from this picture is why conditions~\ref{condn:alq_normalisation} and~\ref{condn:alq_twisted_antisymmetry} should hold for this extended pairing~$h_3$, where in~\ref{condn:alq_normalisation} the $p$-adic dilogarithm is extended to all of $\QQ_p\setminus\{1\}$ in the usual way (see e.g.\ \cite[\S2]{besser_lip_2008}; the definition also depends on a choice of branch of the logarithm). The twisted antisymmetry relation~\ref{condn:alq_twisted_antisymmetry} should follow formally by an argument similar to that of \cite[\S10]{DCW}, but condition~\ref{condn:alq_normalisation} is a lot less obvious to us, since it involves relating the $p$-adic period points of \cite{chatzistamatiou-unver:p-adic_periods} to the definition of the $p$-adic dilogarithm. It is possible that one could prove this by describing the pairing~$h_3$ in terms of the syntomic regulator and using the relation to polylogarithms of \cite[Theorem~1.6]{besser-de_jeu:syntomic_regulators}, but doing so would require checking a number of technical compatibilities which are orthogonal to the main thrust of this paper. It is precisely to avoid tackling these kinds of foundational issues that we preferred to modify the algorithm from \cite{DCW} instead, to avoid factors of~$p$ by hand.

				

			\end{rem}
			
			We now describe how to construct a decomposition~\eqref{eq: wedge decomposition} of $[\ell] \wedge [q]$ as a rational linear combination of Steinberg elements $[t] \wedge [1-t]$ with $t, 1-t \in \ZZ_{(p)}^\times$. We may assume that $\ell < q$. We proceed by induction on $(q, \ell)$, ordered lexicographically. More precisely, we show that $[\ell] \wedge [q]$ can be expressed as a $\QQ$-linear combination of Steinberg elements, terms of the form $[\ell'] \wedge [q']$ with $\ell' < q' < q$, and in the case $\ell > 2$ the particular element $[2] \wedge [q]$. Our algorithm is based on the following observation:

			\begin{lemma}
				Let $\Sigma$ denote the finite set of integers $z$ of absolute value $<q$, together with the even integers $z$ of absolute value $<2q$. Then for all $z_0\in\Sigma$, there is a $z_1\in\Sigma$ such that $q\mid \ell z_0-z_1$ and neither $z_1$ nor $r_1 = \frac{\ell z_0-z_1}q$ is divisible by~$p$.
			\end{lemma}
			
			\begin{proof}
				There are three values $z_1\in \Sigma$ such that $z_1\equiv \ell z_0$ modulo~$q$: two which have absolute value~$<q$ and one which is even and has absolute value in~$(q,2q)$. These values of~$z_1$ form an arithmetic progression of common difference $q$. The corresponding values of $r_1$ form an arithmetic progression of common difference $-1$. It follows that at least one of these values has both $z_1$ and $r_1$ not divisible by~$p$.
			\end{proof}
			
			We use the lemma to find sequences $1=z_0,z_1,\dots$ and $r_1,r_2,\dots$ of integers, all prime to $p$, such that $\lvert z_i\rvert<q$ and
			\[
			qr_i = \ell z_{i-1}-z_i \hspace{0.4cm}\text{or}\hspace{0.4cm} qr_i = \ell z_{i-1}-2z_i
			\]
			for all $i$. (In the case $\ell = 2$ we require $qr_i = \ell z_{i-1}-z_i$ but allow $\lvert z_i \rvert < 2q$ rather than $\lvert z_i \rvert < q$.)
			Every $z_i$ and every $r_i$ have no prime factor $\geq q$, the latter since $\lvert r_i\rvert \leq \frac1q \left(\lvert \ell \rvert \lvert z_{i-1} \rvert + 2 \lvert z_i \rvert \right) \leq\frac1q\left((q-1)^2+2(q-1)\right)<q$.
			
			From this point, we argue as in \cite{DCW}. Define elements $f_i\in\bigwedge^2E$ for $i\geq1$ by
			\[
			f_i \coloneqq [\ell]\wedge[q] + [z_{i-1}]\wedge[q] - [z_i]\wedge[q] \,.
			\]
			Since $\Sigma$ is finite, there must be indices $m<n$ such that $z_m=\pm z_n$, and hence
			\[
			[\ell]\wedge[q] = \frac1{n-m}\cdot\sum_{i=m+1}^nf_i \,.
			\]
			
			We have the identity
			\begin{align}
				\label{eq: f_i without 2}
				f_i &= \left[\tfrac{z_i}{\ell z_{i-1}}\right]\wedge\left[\tfrac{r_i}{\ell z_{i-1}}\right] -\left[\frac{z_i}{\ell z_{i-1}}\right]\wedge\left[1-\frac{z_i}{\ell z_{i-1}}\right] \hspace{0.4cm}\text{or} \\
				\label{eq: f_i with 2}
				f_i &= \left[\tfrac{2z_i}{\ell z_{i-1}}\right]\wedge\left[\tfrac{r_i}{\ell z_{i-1}}\right] -\left[\tfrac{2z_i}{\ell z_{i-1}}\right]\wedge\left[1-\tfrac{2z_i}{\ell z_{i-1}}\right]+[2]\wedge[q] \,,
			\end{align}
			according as $\ell z_{i-1}-qr_i$ is equal to $z_i$ or $2z_i$. In either case, because the prime factors of $r_i, z_i, z_{i-1}$ and $\ell$ are smaller than $q$ and distinct from $p$, $f_i$ can be expressed as a linear combination of smaller basis elements, Steinberg elements, and the particular element $[2]\wedge[q]$.
			
			The element $[2]\wedge[q]$ appears only if $\ell \neq 2$ by construction, so that it can be expressed inductively in terms of Steinberg elements.
			
			An implementation of this algorithm in SageMath is provided in~\cite{dcw_coefficients}.
			
			\begin{rem}
				This algorithm does not give us any control over the number of terms of the resulting decompositions of the generators $[\ell]\wedge[q]$ in $\bigwedge^2 E$. However, using some linear algebra, one can easily get new decompositions with an explicitly bounded number of terms:
				
				Given a positive integer $b$, we consider the subspace $V_b$ of $\bigwedge^2 E$ generated by pairs of primes $[\ell] \wedge [q]$ with $\ell< q< b$ such that $\ell,q \neq p$, its ``canonical basis''. The algorithm produces a generating family of $V_b$ by Steinberg elements $[t] \wedge [1-t]$ where $t,1-t$ are rationals containing only prime factors $< b$ distinct from $p$, whose coordinates in the canonical basis are easy to compute. Inside such a generating family, there is a basis of $V_b$ made up of Steinberg elements, and we can compute the coordinates of the vectors of the canonical basis in this Steinberg basis. This yields decompositions with at most $\dim(V_b)$ terms. The above is essentially the procedure followed by our implementation in Sage. 
			\end{rem}
			
			\subsubsection{Examples}
			
			We give some examples of coefficients $a_{\ell,q}$ for primes $\ell, q$ different from a given odd prime $p$, as well as the corresponding decomposition of $[\ell] \wedge [q]$ in $\bigwedge^2 E$ where $E = \QQ \otimes_{\ZZ} \ZZ_{(p)}^{\times}$.
			
			Note that if $\ell = 2$ and $q$ is a prime of the form $q = 2^n \pm 1$, i.e. a Mersenne or Fermat prime, we have
			\[[2] \wedge [q] = 
			\frac{1}{n} \cdot \left(\left[1\mp q\right] \wedge \left[\pm q\right] \right) , \] 
			yielding the coefficient \[a_{2,q} = \frac12\cdot\log(2)\log(q) - \frac1{2n}\cdot\Bigl(\Li_2(1\mp q)-\Li_2(\pm q)\Bigr) = -\frac1{n} \cdot \Li_2(1 \mp q).\]
			This is the same value we calculated in Lemma~\ref{lem: Fermat Mersenne DCW coefficients} above using the commutativity of the Chabauty--Kim diagram.
			
			
			We give some further examples using the algorithm described above (still fixing $p=3$):
			
			\begin{itemize}
				\item 
				Let $\{\ell,q\}=\{2,11\}$. The algorithm yields
				\[[2] \wedge [11] = -\tfrac{1}{5} \cdot \left[\tfrac{5}{16}\right] \wedge \left[\tfrac{11}{16}\right] + \tfrac{2}{5}\cdot \left[-4\right] \wedge \left[5\right]  + \tfrac{1}{5}\cdot \left[-10\right] \wedge \left[11\right]. \]
				This determines the coefficient $a_{2,11}$ as
				\begin{align*}
					a_{2,11} &= \tfrac12 \log(2)\log(11) + \tfrac{1}{2} \Bigl( -\tfrac15\left(\Li_2\left(\tfrac{11}{16}\right) - \Li_2\left(\tfrac{5}{16}\right)\right)\\ 
					&\quad + \tfrac25 \left(\Li_2\left(5\right) - \Li_2\left(-4\right) \right)
					+ \tfrac15 \left( \Li_2\left(11\right) - \Li_2\left(-10\right)\right)
					\Bigr).  
				\end{align*}
					\item Let $\{\ell,q\} = \{5,7\}$. The algorithm yields
					\begin{align*}
						[5] \wedge [7] = -\tfrac{1}{2} \cdot &[-4] \wedge [5] + \left[-\tfrac{5}{2}\right] \wedge \left[\tfrac{7}{2}\right] - \tfrac{1}{3} \cdot \left[\tfrac{1}{8}\right] \wedge \left[\tfrac{7}{8}\right],
					\end{align*}
					determining the coefficient $a_{5,7}$ as
					\begin{align*}
						a_{5,7} &= \tfrac{1}{2} \log(5) \log(7) - \tfrac{1}{2} \Bigl( -\tfrac12 \left(\Li_2\left(5\right) - \Li_2\left(-4\right) \right)\\
						&\quad + \left(\Li_2\left(\tfrac{7}{2}\right) - \Li_2\left(-\tfrac{5}{2}\right) \right)
						- \tfrac{1}{3} \left(\Li_2\left(\tfrac{7}{8}\right) - \Li_2\left(\tfrac{1}{8}\right)\right) \Bigr).
					\end{align*}
				\end{itemize}
				
				\subsection{Refined Selmer Schemes}
				
				We have seen that, for $\mathcal{X}=\PP^1\setminus\{0,1,\infty\}$ and~$n=2$, the global Selmer scheme $\Sel_{S,2}=\AA^S\times\AA^S$ is $2|S|$-dimensional, while the local Selmer scheme $H^1_f(G_p,U^\et_2)=\AA^3$ is $3$-dimensional. So for~$|S|{}=1$, the Chabauty--Kim inequality~\eqref{eq:c-k_ineq} holds, and so the Chabauty--Kim locus~$\mathcal{X}(\ZZ_p)_{S,2}$ is finite. Using the above explicit description of the Chabauty--Kim diagram in depth~$2$, Dan-Cohen and Wewers gave an explicit description of~$\mathcal{X}(\ZZ_p)_{S,2}$ in this case: it is the vanishing locus of the Coleman function
				\[
				2\Li_2(z)-\log(z)\log(1-z)
				\]
				(independent of~$S$) \cite[\S12]{DCW}.
				
				For~$|S|{}=2$, however, the dimension inequality~\eqref{eq:c-k_ineq} fails, and the Chabauty--Kim locus~$\mathcal{X}(\ZZ_p)_{S,2}$ is not finite (at least in general). To circumvent this issue we use a certain refinement of the Chabauty--Kim method, suggested by the second author and Netan Dogra, which replaces the Selmer scheme $\Sel_{S,n}$ with a smaller \emph{refined Selmer scheme}~$\Sel_{S,n}^\min$. We recall the definition, in slightly greater generality than \cite[Definition~1.2.2]{BD19}.

				\begin{defn}\label{def:refined_selmer_scheme}
					Let~$p$ be a prime and~$n\geq0$ a non-negative integer. Let~$X/\QQ$ be a smooth hyperbolic curve. \begin{enumerate}
						\item A \emph{Selmer structure} for $X$ is a collection of sets $(\mathcal{X}_{\ell})_{\ell}$ for every prime number $\ell$, such that, for every $\ell$, $\mathcal{X}_{\ell} \subset X(\QQ_{\ell})$, and, for all but finitely many $\ell$, ~$\mathcal{X}_\ell$ is the set of $\ZZ_\ell$-integral points on the good model\footnote{By the ``good model'' of~$X$ over~$\ZZ_\ell$, we mean the~$\ZZ_\ell$-scheme $\mathcal{X}/\ZZ_\ell$ which is the complement of an \'etale divisor in a smooth proper $\ZZ_\ell$-scheme, together with an isomorphism~$\mathcal{X}_{\QQ_\ell} \cong X_{\QQ_\ell}$. The good model of~$X$, if it exists, is unique up to unique isomorphism.} of~$X$ over~$\ZZ_\ell$.
						\item If~$b$ is a $K$-rational basepoint (possibly tangential) and~$U_n^\et$ the $\QQ_p$-pro-unipotent \'etale fundamental group of~$(X,b)$, truncated in depth~$n$, then we define the \emph{refined Selmer scheme} associated to~$(X,(\mathcal{X}_\ell)_{\ell})$ to be the subscheme\footnote{There is again a small subtlety here, since the cohomology functor $H^1(G_\QQ,U_n^\et)$ is not representable. However, one can show that the refined Selmer scheme is contained inside $H^1(G_T,U_n^\et)$ where~$G_T$ is the largest quotient of~$G_\QQ$ unramified outside a sufficiently large finite set of primes~$T$ \cite[Proposition~3.2.4]{betts:effective}. This latter cohomology functor is representable by a $\QQ_p$-scheme of finite type, so it makes sense to talk of its subschemes.}
						\[
						\Sel_{\mathcal{X},n}^\min \subseteq H^1(G_\QQ,U_n^\et)
						\]
						parametrising those cohomology classes~$\xi$ whose restriction to a decomposition group~$G_\ell$ at a prime~$\ell$ lies in the Zariski-closure of the image of the local Kummer map
						\[
						j_\ell\colon \mathcal{X}_\ell \to H^1(G_\ell,U_n^\et)
						\]
						for all primes~$\ell$.
						\item We define the \emph{refined Chabauty--Kim locus} associated to~$(X,(\mathcal{X}_\ell)_{\ell})$ to be the subset
						\[
						\mathcal{X}_{p,n}^{\min} \subseteq \mathcal{X}_p
						\]
						consisting of those points $z\in\mathcal{X}_p$ such that $j_p(z)$ lies in the scheme-theoretic image of the localization map $\Sel_{\mathcal{X},n}^\min\to H^1(G_p,U_n^\et)$.
					\end{enumerate}
				\end{defn}

				We will primarily be interested in the following case. Suppose that~$S$ is a finite set of primes and that~$\mathcal{X}/\ZZ_S$ is a model of the hyperbolic curve~$X$ which is the complement of a horizontal divisor $D$ in a proper regular~$\ZZ_S$-scheme $Y$. We can then define the \emph{natural Selmer structure} $(\mathcal{X}_\ell)_{\ell}$ by
				\[
				\mathcal{X}_\ell \coloneqq 
				\begin{cases}
					\mathcal{X}(\ZZ_\ell) & \text{if $\ell\notin S$,} \\
					X(\QQ_\ell) & \text{if $\ell\in S$.}
				\end{cases}
				\]
				The resulting refined Selmer scheme is denoted by $\Sel_{S,n}^{\min}(\mathcal{X})$ (or simply~$\Sel_{S,n}^{\min}$ if~$\mathcal{X}$ is understood), and the resulting refined Chabauty--Kim locus for~$p\notin S$ by $\mathcal{X}(\ZZ_p)_{S,n}^{\min}$. It follows from the definition that one has the inclusion
				\[
				\mathcal{X}(\ZZ_S) \subseteq \mathcal{X}(\ZZ_p)_{S,n}^{\min} \subseteq \mathcal{X}(\ZZ_p) .
				\]
				If moreover $Y$ is a smooth $\ZZ_S$-scheme, $D$ is \'etale and the basepoint~$b$ is $\ZZ_S$-integral, then $\Sel_{S,n}^{\min}$ is a closed subscheme of the usual Selmer scheme $\Sel_{S,n}=H^1_{f,S}(G_\QQ,U_n^{\et})$, and hence 
				\[
				\mathcal{X}(\ZZ_S) \subseteq \mathcal{X}(\ZZ_p)_{S,n}^{\min} \subseteq \mathcal{X}(\ZZ_p)_{S,n} \subseteq \mathcal{X}(\ZZ_p) .
				\]
				
				Given a Selmer structure $(\mathcal{X}_{\ell})_{\ell}$, we stress that $\mathcal{X}_{\ell}$ does not have to contain $\mathcal{X}(\ZZ_S)$. In fact, in the case~$\mathcal X=\PP^1\setminus\{0,1,\infty\}$, we will split the points of~$\mathcal X(\ZZ_S)$ into finitely many subsets and consider appropriate Selmer structures for each one of these subsets. So we end up with a finite set of refined Selmer schemes, each one containing only some of the~$\ZZ_S$-points, and such that the union of the associated refined loci contains all of the~$\ZZ_S$-points. 
				
				\begin{rem}\label{rem:empty_locus}
					If the model~$\mathcal X$ above has $\mathcal X(\ZZ_\ell)=\emptyset$ for some~$\ell\notin S$ or $X(\QQ_\ell)=\emptyset$ for some~$\ell\in S$, then the refined Selmer scheme $\Sel_{S,n}^{\min}(\mathcal X)$ is the empty scheme, and so the refined Chabauty--Kim locus $\mathcal X(\ZZ_p)_{S,n}^{\min}$ is likewise empty for any~$n$ and any~$p\notin S$. Thus the equality
					\[
					\mathcal X(\ZZ_S) = \mathcal X(\ZZ_p)_{S,n}^{\min}
					\]
					holds automatically in such cases, since both sides are the empty set. In the particular case that $\mathcal X=\PP^1_{\ZZ_S}\setminus\{0,1,\infty\}$, this is saying that the refined version of Kim's Conjecture holds automatically whenever $S\not\ni2$, as we asserted in the introduction.
					
					We remark that something similar holds for the Selmer scheme as defined in \cite[\S8.1]{BDCKW}: it is empty whenever $\mathcal X(\ZZ_\ell)=\emptyset$ for some $\ell\notin S\cup\{p\}$. In particular, the unrefined version of Kim's Conjecture in \cite[Conjecture~3.1 \& \S8.1]{BDCKW} holds automatically for $\mathcal X=\PP^1_{\ZZ_S}\setminus\{0,1,\infty\}$ whenever $S\not\ni2$.
				\end{rem}

				In the particular case that~$\mathcal X=\PP^1\setminus\{0,1,\infty\}$, $n \leq 2$ and~$p\notin S$, this can all be made very explicit. Since $H^1(G_\ell,\QQ_p(1))=\QQ_p$ by Kummer theory and $H^1(G_{\ell}, \QQ_p(2)) = 0$, we have $H^1(G_\ell,U^\et_n)=\AA^2$ for $\ell \in S$ and $n \leq 2$, and the local Kummer map $j_\ell\colon\mathcal{X}(\QQ_\ell)\to\QQ_p^2$ is given by $z\mapsto(v_\ell(z),v_\ell(1-z))$. So the Zariski-closure of the image is as follows.
				
				\begin{lemma}\label{lem:three}
					The Zariski-closure of $j_\ell(\mathcal{X}(\QQ_\ell))$ in $\AA^2$ is the union of the three lines $x=0$, $y=0$ and $x=y$.
					\begin{proof}
						For~$z\in\mathcal{X}(\QQ_\ell)$, the ultrametric triangle inequality gives that
						\[
						\min\{v_\ell(z),v_\ell(1-z)\}\leq0
						\]
						with equality if $v_\ell(z)\neq v_\ell(1-z)$. So we have either $v_\ell(1-z)=0$ or $v_\ell(z)=0$ or $v_\ell(z)=v_\ell(1-z)$, i.e.\ $j_\ell(z)=(v_\ell(z),v_\ell(1-z))$ lies on one of the three lines mentioned. It is easy to see that the image is Zariski-dense in this union, e.g.\ the points $j_\ell(\ell^m)=(m,0)$ for~$m\geq1$ are Zariski-dense in the line~$y=0$.
					\end{proof}
				\end{lemma}
				
				\begin{rem}
					In Lemma~\ref{lem:three} we see a shadow of the tropicalization (or Berkovich analytification) of the thrice-punctured line~$\dX$. The image of~$j_\ell$ is contained in~$\QQ^2$, and the closure of~$j_\ell(\dX(\Qbar_\ell))$ inside~$\RR^2$ is the tropicalization of $\dX$ by Kapranov's Theorem \cite[Theorem~3.1.3]{maclagan-sturmfels:tropical_geometry} (viewing~$\dX$ as the toric hypersurface in~$\mathbb{G}_m^2$ cut out by the equation $z_1+z_2=1$). This tropicalization is the tropical line, i.e.\ the union of three rays through the origin as shown below, and these three rays correspond to the three irreducible components of the Zariski-closure of~$j_\ell(\dX(\QQ_\ell))$ discussed in Lemma~\ref{lem:three}.

				\end{rem}
				
				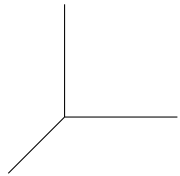
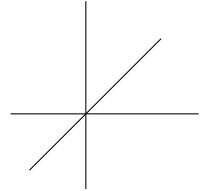
\begin{figure}[!h]
					\centering
					\begin{subfigure}[b]{0.4\textwidth}
						\centering
						\begin{tikzpicture}
							\draw (0,0) -- (1.5,0);
							\draw (0,0) -- (0,1.5);
							\draw (0,0) -- (-0.75,-0.75);
							\draw[white] (-0.75,-0.75) -- (-1.0,-1.0);
						\end{tikzpicture}
						\caption{The tropicalization of the thrice-punctured line, inside~$\RR^2$.}
					\end{subfigure}
					\hfill
					\begin{subfigure}[b]{0.4\textwidth}
						\centering
						\begin{tikzpicture}
							\draw (-1.,0) -- (1.5,0);
							\draw (0,-1.0) -- (0,1.5);
							\draw (1.0,1.0) -- (-0.75,-0.75);
						\end{tikzpicture}
						\caption{The Zariski-closure of the image of~$j_\ell$ inside~$\AA^2_{\QQ_p}$.}
					\end{subfigure}
					\caption{The three components of the Zariski-closure of the image of~$j_\ell$ correspond to the rays of the tropicalization.}
				\end{figure}
				
				Using this, we see that the refined Selmer scheme $\Sel_{S,n}^\min$ is either empty or the union of $3^{|S|}$ subspaces. Specifically, let us define
				\begin{equation}
					\label{eq:refinedconditions}
					p_{\diagup}(x,y) = x-y, \quad p_|(x,y)=x, \quad p_{-}(x,y)=y,
				\end{equation}
				so that the locus $p_i(x,y)=0$ defines a line in $\AA^2$ (the subscripts $\diagup,|,-$ depict the direction of the corresponding line). Then we have the following.
				
				\begin{lemma}\label{lem:nine}
					Let $n \leq 2$. If~$2\notin S$ then the refined Selmer scheme $\Sel_{S,n}^\min\subseteq\Sel_{S,n}=\AA^S\times\AA^S$ is empty. Otherwise, it is equal to the union of the subspaces
					\[\Sel^{\Sigma}_{S,n} = \{ ((x_{\ell})_{\ell \in S}, (y_{\ell})_{\ell \in S}) : p_{i_{\ell}}(x_{\ell}, y_{\ell}) = 0 \; \forall \ell \in S\} \subseteq \mathbb A^{S} \times \mathbb A^{S}\]
					for the $3^{\lvert S\rvert }$ choices of tuples of conditions
					\[\Sigma = (i_{\ell})_{\ell \in S} \in\{\diagup, |,-\}^{ S}\text.\]
					Each of these subspaces is $| S|$-dimensional.
				\end{lemma}
				
				
				
				The decomposition of the refined Selmer scheme~$\Sel_{S,n}^\min$ into the subschemes $\Sel_{S,n}^\Sigma$ in Lemma~\ref{lem:nine} induces a corresponding decomposition of the refined Chabauty--Kim locus~$\mathcal{X}(\ZZ_p)_{S,n}^\min$.
				
				\begin{defn}
					Let $n \leq 2$ and let $\Sigma = (i_\ell)_{\ell \in S} \in \{\diagup,|,-\}^S$ be a choice of refinement conditions for each $\ell \in S$. Denote by
					\[
					\mathcal X(\ZZ_p)_{S,n}^\Sigma \subseteq \mathcal X(\ZZ_p)
					\]
					the set of points $z\in\mathcal{X}(\ZZ_p)_{S,n}$ such that~$j_p(z)$ lies in the scheme-theoretic image of the localization map $\loc_p\colon\Sel_{S,n}^\Sigma\to H^1(G_p,U_n^\et)$. It follows from Lemma~\ref{lem:nine} that the refined Chabauty--Kim locus $\mathcal{X}(\ZZ_p)_{S,n}^\min$ is empty if~$2\notin S$, and otherwise admits a decomposition
					\[
					\mathcal X(\ZZ_p)_{S,n}^{\min} = \bigcup_{\Sigma}\mathcal X(\ZZ_p)_{S,n}^{\Sigma}.
					\]
				\end{defn}

				\begin{rem}
					\label{rem: refined subsets}
					The scheme $\Sel_{S,n}^\Sigma$ and set $\mathcal{X}(\ZZ_p)_{S,n}^\Sigma$ are the refined Selmer scheme and refined Chabauty--Kim locus corresponding to the Selmer structure
					\[
					\mathcal{X}_\ell^{\Sigma} \coloneqq 
					\begin{cases}
						\mathcal{X}(\ZZ_\ell) & \text{if $\ell\notin S$,} \\
						\{z\in X(\QQ_\ell) \::\: z\not\equiv0,1 \text{ mod }\ell\} & \text{if $\ell\in S$ and $i_\ell=\diagup$,} \\
						\{z\in X(\QQ_\ell) \::\: z\not\equiv0,\infty \text{ mod }\ell\} & \text{if $\ell\in S$ and $i_\ell={|}$,} \\
						\{z\in X(\QQ_\ell) \::\: z\not\equiv1,\infty \text{ mod }\ell\} & \text{if $\ell\in S$ and $i_\ell={-}$.}
					\end{cases}
					\]
				\end{rem}


				
				\subsection{The $S_3$-action}
				
				It turns out that when trying to compute the refined Chabauty--Kim locus~$\mathcal{X}(\ZZ_p)_{S,n}^\min$, it is more efficient to compute the sets $\mathcal{X}(\ZZ_p)_{S,n}^\Sigma$ separately and then take the union over~$\Sigma$. \emph{A priori} this involves computing $3^{|S|}$ different sets, but we will leverage the natural action of~$S_3$ on~$\mathcal{X}$ to reduce this number significantly. We start by proving a general functoriality statement for refined Chabauty--Kim loci, building on the corresponding statement for unrefined Chabauty--Kim loci \cite[\S2.9]{BDCKW}.

				\begin{prop}[Functoriality of refined Chabauty--Kim loci]
					\label{thm: functoriality}
					Let~$p$ be a prime and~$n\geq 1$ a positive integer.
					\begin{enumerate}
						\item \label{item: base point independence of refined CK locus}
						Let~$X/\QQ$ be a smooth hyperbolic curve and let $(\mathcal{X}_{\ell})_{\ell}$ be a Selmer structure. Then the refined Chabauty--Kim locus $\mathcal{X}_{p,n}^\min$ is independent of the choice of basepoint~$b$.
						\item \label{item: functoriality of CK locus}
						Let~$Y/\QQ$ be another smooth hyperbolic curve with a Selmer structure $(\mathcal{Y}_{\ell})_{\ell}$. Suppose that~$f\colon X\to Y$ is a morphism of $\QQ$-varieties such that $f(\mathcal{X}_\ell)\subseteq \mathcal{Y}_\ell$ for all~$\ell$. Then
						\[
						f(\mathcal{X}_{p,n}^\min) \subseteq \mathcal{Y}_{p,n}^\min \,.
						\]
					\end{enumerate}
				\end{prop}
				
				\begin{proof}
					For the purpose of this proof, let us denote by $U^{\et}(b)$ the $\QQ_p$-pro-unipotent étale fundamental group of $X$ with base point $b$, and by $U_n^{\et}(b)$ its quotient of unipotency depth~$n$. Suppose $c$ is a second basepoint. We claim to have a diagram for all primes $\ell$ as follows:
					\[
					\begin{tikzcd}
						& & \mathcal{X}_{\ell} \dlar["j_{\ell}^{(b)}"] \arrow[ddl, bend left, "j_{\ell}^{(c)}"]\\
						H^1(G_{\QQ}, U_n^{\et}(b)) \rar["\loc_{\ell}"] \dar["\wr"] & H^1(G_{\ell}, U_n^{\et}(b)) \dar["\wr"] \\
						H^1(G_{\QQ}, U_n^{\et}(c)) \rar["\loc_{\ell}"] & H^1(G_{\ell}, U_n^{\et}(c)).
					\end{tikzcd}
					\]
					Here, $j_{\ell}^{(b)}$ denotes the local Kummer map which maps $x \in \mathcal{X}_{\ell}$ to the class representing the torsor of paths from $b$ to $x$. For the full fundamental group (``$n = \infty$''), it is shown in \cite[§2.9]{BDCKW} that there are canonical isomorphisms $H^1(G, U^{\et}(b)) \cong H^1(G, U^{\et}(c))$ for $G = G_{\QQ}$ and $G = G_{\ell}$, compatible with the localization and the Kummer maps. In terms of $G$-equivariant $U^{\et}(b)$-torsors and $U^{\et}(c)$-torsors, the isomorphisms are given by twisting with the $\QQ_p$-pro-unipotent étale path space $P^{\et}(c,b)$, which is a $G$-equivariant $U^{\et}(b)$-$U^{\et}(c)$-bitorsor\footnote{The authors of loc.\ cit. write $P^{\et}(b,c)$, not $P^{\et}(c,b)$, which seems to be a mistake.}. 
					The finite depth variant is given by twisting with the corresponding quotient $P^{\et}_n(c,b)$. 
					
					Given the diagram, it follows from the definitions that the change of basepoint isomorphism on the left maps the refined Selmer scheme $\Sel_{S,n}^{\min}(\mathcal{X})$ at~$b$ isomorphically to the one at~$c$. It then follows that both result in the same refined Chabauty--Kim locus, proving~\ref{item: base point independence of refined CK locus}.
					
					To show~\ref{item: functoriality of CK locus}, denote by $U_n^{\et}$ and $V_n^{\et}$ the $\QQ_p$-pro-unipotent étale fundamental groups of $(X,b)$ and $(Y, f(b))$, respectively, truncated in unipotency depth~$n$. Thus, the choice of basepoint $f(b)$ on $Y$ depends on the choice of basepoint $b$ on $X$. By~\ref{item: base point independence of refined CK locus}, the refined Chabauty--Kim locus is not affected by this choice. The map $f\colon X \to Y$ of $\QQ$-varieties induces a $G_{\QQ}$-equivariant homomorphism $f_*\colon U_n^{\et} \to V_n^{\et}$ of fundamental groups, which in turn induces maps $f_*\colon H^1(G, U_n^{\et}) \to H^1(G, V_n^{\et})$ on non-abelian cohomology functors for $G = G_{\QQ}$ and $G_{\ell}$ for all primes~$\ell$. These fit into a commutative diagram as follows,
					\[
					\begin{tikzcd}
						& & \mathcal{X}_{\ell} \dlar["j_{\ell}^X"] \dar["f"] \\
						H^1(G_{\QQ}, U_n^{\et}) \rar["\loc_{\ell}"] \dar["f_*"] & H^1(G_{\ell}, U_n^{\et}) \dar["f_*"] & \mathcal{Y}_{\ell} \dlar["j_{\ell}^Y"] \\
						H^1(G_{\QQ}, V_n^{\et}) \rar["\loc_{\ell}"] & H^1(G_{\ell}, V_n^{\et}),
					\end{tikzcd}
					\]
					with $j_{\ell}^X$ and $j_{\ell}^Y$ denoting the respective local Kummer maps. Using the diagram, it again follows from the definitions that the vertical map on the left maps the refined Selmer scheme $\Sel_{S,n}^{\min}(\mathcal{X})$ of $X$ into the refined Selmer scheme $\Sel_{S,n}^{\min}(\mathcal{Y})$ of $Y$, and this implies $f(\mathcal{X}_{p,n}^{\min}) \subseteq \mathcal{Y}_{p,n}^{\min}$.
				\end{proof}
				
				Let us apply this to the automorphisms of the thrice-punctured line given by the natural $S_3$-action. The action is generated by the two automorphisms $z \mapsto 1-z$ and $z \mapsto 1/z$ which permute the three cusps $\{0,1,\infty\}$. The full group is given by the following rational functions:
				\[ z, \; 1-z,\; \frac{1}{z},\; \frac{z-1}{z},\; \frac{z}{z-1},\; \frac{1}{1-z}. \]
				
				There is also a natural action of $S_3$ on $\{\diagup, |, -\}$: $z \mapsto 1-z$ interchanges $| \leftrightarrow -$ and $z \mapsto 1/z$ interchanges $\diagup \leftrightarrow -$. This action gives an isomorphism between $S_3$ and the group of self-permutations of $\{\diagup, |, -\}$ (e.g.\ because the image contains two transpositions), in particular it is $3$-transitive.

				\begin{cor}
					\label{cor: S3-action on CK locus}
					Assume $2 \in S$, let $\mathcal X = \PP^1_{\ZZ_S} \setminus \{0,1,\infty\}$ and let $p \not\in S$. 
					\begin{enumerate}
						\item \label{item: locus S3-stable}
						The refined Chabauty--Kim locus $\mathcal{X}(\ZZ_p)_{S,n}^{\min}$ for the natural Selmer structure is stable under the $S_3$-action for any $n \geq 1$. 
						\item \label{item: S3 permutes refined loci}
						For $n\leq 2$, the subsets $\mathcal X(\ZZ_p)_{S,n}^{\Sigma}$ in the union $\mathcal{X}(\ZZ_p)_{S,n}^{\min} = \bigcup_{\Sigma}\mathcal X(\ZZ_p)_{S,n}^{\Sigma}$, with $\Sigma \in \{\diagup, |,-\}^S$ running through tuples of refinement conditions, are permuted by the $S_3$-action, in the sense that 
						\[ \sigma(\mathcal X(\ZZ_p)_{S,n}^{\Sigma}) = \mathcal X(\ZZ_p)_{S,n}^{\sigma(\Sigma)}. \]
					\end{enumerate}
				\end{cor}

				\begin{proof}
						Part~\ref{item: locus S3-stable} follows from Proposition~\ref{thm: functoriality}\,\ref{item: functoriality of CK locus} with the natural Selmer structure defined earlier, 
						since these sets are preserved by the $S_3$-action. Given a choice of refinement conditions $\Sigma = (i_{\ell})_{\ell \in S} \in \{\diagup, |,-\}^S$ and defining $\mathcal X_{\ell}^\Sigma$ as in Remark~\ref{rem: refined subsets}, we have $\sigma(\mathcal X_{\ell}^\Sigma) = \mathcal X_{\ell}^{\sigma(\Sigma)}$ for all primes $\ell$. This implies~\ref{item: S3 permutes refined loci} again by Proposition~\ref{thm: functoriality}\,\ref{item: functoriality of CK locus}.
					\end{proof}
					

					\section{Explicit Equations}
					\label{explicit equations}
					
					The refined Selmer scheme~$\Sel_{S,2}^\min$ for $\mathcal{X}=\PP^1\setminus\{0,1,\infty\}$ is either empty or $|S|$-dimensional by Lemma!\ref{lem:nine}, so the localization map $\Sel_{S,2}^\min\to H^1_f(G_p,U^\et_2)$ has non-dense image and the refined Chabauty--Kim locus $\mathcal{X}(\ZZ_p)_{S,2}^\min$ is finite as soon as $|S|{}\leq2$. Using the explicit descriptions of the Chabauty--Kim diagram from Section~\ref{subsec: explicit_c-k_depth_2}--\ref{subsec:locdepth2}, we will compute explicit equations for the refined locus in this case.
					
					Since $\mathcal{X}(\ZZ_p)_{S,n}$ is empty whenever $2\notin S$, we will restrict attention in this section to sets~$S$ containing~$2$.
					
					
					\subsection{The case $S = \{2\}$}
					
					Assume that $S=\{2\}$ and fix a prime $p \neq 2$. The thrice-punctured line has the $S$-integral points $2,-1,\frac12$. Since $S_3$ acts transitively on the set $\{\diagup, |, -\}$ it suffices to consider a single refined Selmer condition, say $|$. We start by working in depth $n = 1$. We have
					\[ \Sel_{\{2\},1} \cong \AA^2 \]
					with localization map $\loc_p: \Sel_{\{2\},1} \to \AA^2$ given by
					\[ \loc_p(x_2, y_2) = \begin{pmatrix} \log(2) x_2\\
						\log(2) y_2
					\end{pmatrix}.\]
					Since the image is Zariski dense, the unrefined Chabauty--Kim method does not apply. However, the refined Selmer subspace $\Sel_{\{2\},1}^{|} \subseteq \Sel_{\{2\},1}$, which as in \eqref{eq:refinedconditions} is cut out by the equation $x_{2} = 0$, is one-dimensional,
					and the localization map restricts as
					\[ \loc_p(0,y_2) = \begin{pmatrix} 0\\
						\log(2) y_2
					\end{pmatrix}.\]
					In the coordinates $u,v$, the image is cut out by the equation $u = 0$, which becomes
					\[ \log(z) = 0 \]
					after pulling back along $j_{\dR}$. We conclude:
					
					\begin{prop}\label{prop:depth_1}
						The refined Chabauty--Kim method in depth~$1$ with any odd prime~$p$ shows the finiteness of $\mathcal X(\ZZ[1/2])$. More precisely, the refined set $\mathcal X(\ZZ_p)_{\{2\},1}^{\min}$ consists of the nontrivial $(p-1)$-st roots of unity, along with their orbits under the $S_3$-action. \qed
					\end{prop}
					
					
					Choosing $p=3$, we obtain the following:
					
					\begin{cor}
						\label{refined CK conjecture in depth 1}
						We have $\mathcal{X}(\ZZ_3)^{\min}_{\{2\},1} = \{-1,2,1/2\}$, i.e.\ the refined Chabauty--Kim conjecture holds in depth~1 for $S = \{2\}$ and $p = 3$.
					\end{cor}
					
					\begin{rem}
						As discussed above, the original formulation of Kim's Conjecture does not hold in depth~$1$ in this case: one needs to either decrease the size of $S$, as in \cite[§6]{BDCKW}, or go to higher depth, as in \cite[§12]{DCW}.
					\end{rem}
					
					\begin{rem}
						As was pointed out to us by the referee, the unrefined Chabauty--Kim method in depth~$n=1$ for $\PP^1\setminus\{0,1,\infty\}$ is closely related to Skolem's $p$-adic method for solving Thue equations. That is, in the case $S=\{2\}$, we are trying to solve the exponential Diophantine equation
						\begin{equation}\label{eq:exponential_diophantine}
							\pm 2^a\pm 2^b = 1
						\end{equation}
						for $a,b\in\ZZ$. Skolem's method amounts to first solving this equation $p$-adically (for some odd~$p$), observing that we can write
						\[
						\pm2^a = \zeta\cdot\exp(a\log(2)) \quad\text{and}\quad \pm2^b = \zeta\cdot\exp(b\log(2))
						\]
						for some roots of unity~$\zeta,\eta$ in~$\ZZ_p$, where~$\log$ and~$\exp$ denote the usual $p$-adic power series. So solutions to~\eqref{eq:exponential_diophantine} in~$\ZZ$ give rise to solutions of
						\begin{equation}\label{eq:skolem}
							\zeta\cdot\exp(a\log(2)) + \eta\cdot\exp(b\log(2)) = 1
						\end{equation}
						over~$\ZZ_p$. One might hope to solve this latter equation $p$-adically, and thereby derive constraints on the solutions to~\eqref{eq:exponential_diophantine}. In this case, though, there are uncountably many $p$-adic solutions to~\eqref{eq:skolem} and so this unrefined Skolem's method fails to tell us much about the solutions to~\eqref{eq:exponential_diophantine}.
						
						However, once one makes the observation that either $a=0$ or $b=0$ or $a=b$ (from considering~\eqref{eq:exponential_diophantine} $2$-adically), then we are reduced to solving three one-variable exponential Diophantine equations, for example the equation
						\[
						\pm 1 \pm 2^b = 1 \,.
						\]
						Of course, solving this equation is rather trivial, but one can still apply Skolem's method. One ends up wanting to solve the equation
						\[
						\zeta + \eta\cdot\exp(b\log(2)) = 1
						\]
						for $\zeta,\eta$ roots of unity in~$\ZZ_p$ and~$b\in\ZZ_p$. This latter equation has no solution if~$\zeta=1$, and all other values of~$\zeta$ lead to a unique solution $(\eta,b)$. So the constraints on~$\mathcal X(\ZZ[1/2])$ coming from Skolem's method are exactly those stated in Proposition~\ref{prop:depth_1}.
						
						We remark that the application of Chabauty--Kim to the study of more general kinds of exponential Diophantine equations was already investigated in \cite[\S12.2]{DCW}.
					\end{rem}
					
					
					
					
					Now we carry out refined Chabauty--Kim in depth $n = 2$. We have $\Sel_{\{2\},2} \cong \Sel_{\{2\},1} \cong \AA^2$ and the localization map $\loc_p: \Sel_{\{2\},2} \to \AA^3$ is given by
					\[ \loc_p(x_2, y_2) = \begin{pmatrix} \log(2) x_2\\
						\log(2) y_2\\
						\tfrac12 \log(2)^2 x_2 y_2
					\end{pmatrix},\]
					using the twisted antisymmetry relation for $a_{22}=\frac12\log(2)^2$. On the refined Selmer subspace $\Sel_{\{2\},2}^{|}$, the localization map restricts as
					\[\loc_p(0, y_2) = \begin{pmatrix}
						0 \\
						\log(2) y_2\\
						0
					\end{pmatrix},\]
					so that the set $\mathcal X(\ZZ_p)_{\{2\},2}^{|}$ is cut out by the two equations
					\[ \log(z) = 0, \quad \Li_2(z) = 0. \]
					This shows:
					
					\begin{prop}[= Theorem~\ref{thm:main_s=1}]
						\label{prop:refinedCKset_card=1}
						Let $p \neq 2$ be prime. The refined Chabauty--Kim set $\mathcal X(\ZZ_p)^{\min}_{\{2\},2}$ in depth $2$ consists of the nontrivial $(p-1)$-st roots of unity $\zeta \in \ZZ_p$ for which $\Li_2(\zeta) = 0$, along with their $S_3$-orbits. \qed
					\end{prop}
					
					\begin{rem}\label{rem:one_prime_checked}
						The refined version of Kim's Conjecture for~$\mathcal X=\PP^1_{\ZZ}\setminus\{0,1,\infty\}$ and $S=\{2\}$ in depth~$n=2$ is equivalent to the assertion that the only nontrivial $(p-1)$-st root of unity~$\zeta\in\ZZ_p$ for which~$\Li_2(\zeta)=0$ is $\zeta=-1$. For any fixed~$p$, this can be checked on a computer: work of Besser--de Jeu allows one to compute the values~$\Li_2(\zeta)$ for roots of unity~$\zeta\notin\{\pm1\}$ up to any desired $p$-adic precision, and thereby verify that~$\Li_2(\zeta)\neq0$ for these~$\zeta$.
						
						In practice, verifying the refined version of Kim's Conjecture this way is rather slow, since it requires computing $O(p)$ values of the $p$-adic dilogarithm. One can speed up the computation using work of Besser \cite[Prop.\ 2.1, 2.2]{besser_finite_2002} shows that for any $(p-1)$-st root of unity $\zeta\ne 1$ in $\QQ_p$ we have 
						$$ \frac{p^2-1}{p^2}\Li_2(\zeta) \in \ZZ_p$$
						and that this is congruent mod $p$ to 
						$$(1-\zeta^p)^{-1} \li_2(\zeta)$$
						where $\li_2(z)$ is a \emph{finite polylogarithm} function
						\begin{align*}
							\li_n \colon \FF_p &\to \FF_p       \\
							z &\mapsto \sum_{k=1}^{p-1} \frac{z^k}{k^n}\text.
						\end{align*}
						Hence a sufficient condition for $\Li_2(\zeta)$ to be non-zero is that the finite polylogarithm $\li_2(\zeta)$ is non-zero in~$\FF_p$. This gives a much quicker way to verify the refined version of Kim's Conjecture in depth~$2$. One runs through the roots of unity in~$\FF_p$ other than~$\pm1$ (i.e.\ the elements $2,3,\dots,p-2$), checking whether~$\li_2(\zeta)=0$ or not. If it is non-zero, then certainly $\Li_2(\zeta)\neq0$, and for the remaining roots of unity~$\zeta$ one falls back to computing~$\Li_2(\zeta)$ $p$-adically using~\cite{besser_lip_2008} -- heuristically, this should only happen for $O(1)$ values of~$\zeta$. Using this approach, we have verified that the refined version of Kim's Conjecture holds in depth~$n=2$ for all~$3\leq p\leq10^5$.
						
						Pushing this computation further would be possible but this procedure seems to take $O(p^3)$ time to run in practice.
						The only points in this computation where computing $\Li_2(\zeta)$ to more more than 4 digits of $p$-adic precision was necessary were for $p=1093, 3511$ (the known Wieferich primes) and $\zeta$ the root of unity reducing to $2$.
					\end{rem}
					
					
					\begin{rem}
						Already the unrefined Chabauty--Kim method in depth 2 proves the finiteness of $\{2\}$-integral points of $\PP^1 \setminus \{0,1,\infty\}$, since the image of $\loc_p$ is a two-dimensional subspace in $\AA^3$. As in ~\cite[\S 12]{DCW} , the set $\mathcal X(\ZZ_p)_{\{2\},2}$ is the vanishing locus of the function
						\[ 2 \Li_2(z) - \log(z) \log(1-z). \]
						This cuts out precisely the set $\{2,-1,\frac12\}$ of $\{2\}$-integral points for $p = 3,5,7$, but for $p = 11$ one gets additionally the $S_3$-orbit of the point $\frac12 (1 \pm \sqrt5)$. The refined Chabauty--Kim method is able to rule out this point already in depth one. 
					\end{rem}
					
					\subsection{Sets $S$ of size two}
					
					
					Assume now that $S = \{2,q\}$ for some odd prime~$q$. Then each refined Selmer scheme $\Sel_{S,2}^{i,j}$ has dimension $2$, hence the image in $\AA^3$ under $\loc_p$ is non-dense, so that the corresponding vanishing ideal $\mathcal I^{i,j}_{S,2} \neq 0$. As before, we do not need to determine the vanishing ideal for all nine possible refining Selmer conditions $(i,j)\in\{\diagup, |, -\}^2$. The action of~$S_3$ on $\{\diagup, |, -\}$ is $2$-transitive, so there are exactly two $S_3$-orbits in $\{\diagup, |, -\}^2$: the orbit of $(|,|)$ and the orbit of $(|,-)$. So by Corollary~\ref{cor: S3-action on CK locus}\ref{item: S3 permutes refined loci}, $\mathcal X(\ZZ_p)_{S,2}^{\min}$ is the union of the $S_3$-translates of the two refined loci $\mathcal X(\ZZ_p)_{S,2}^{|,|}$ and $\mathcal X(\ZZ_p)_{S,2}^{|,-}$, which we now compute.
					
					For this, recall from \S\ref{subsec:locdepth2} that the localization map with respect to the coordinates $x = (x_2, x_q)$, $y = (y_2, y_q)$ on $\Sel_{S,2} = \AA^2 \times \AA^2$ has the form
					\[ \loc_p(x,y) = \begin{pmatrix}
						\log(\ell) x_2 + \log(q) x_q\\
						\log(\ell) y_2 + \log(q) y_q\\
						\tfrac12 \log(2)^2 x_2 y_2 + a_{2,q} x_2 y_q + a_{q,2} x_q y_2 + \tfrac12 \log(q)^2 x_q y_q
					\end{pmatrix}
					\]
					where $a_{2,q}$, $a_{q,2} \in \QQ_p$ are the coefficients of the pairing~$h_3$.
					
					We first determine the equations for $\mathcal X(\ZZ_p)_{S,2}^{|,|}$.  
					The restriction of $\loc_p$ to the subspace $\Sel_{S,2}^{|,|}$ is given by
					\[ 
					\loc_p(0, 0, y_2, y_q) = \begin{pmatrix}
						0\\
						\log(2) y_2 + \log(q) y_q \\
						0
					\end{pmatrix}.
					\]
					In the coordinates $u,v,w$ on $\AA^3$, the image of $\Sel_{S,2}^{|,|}$ is therefore cut out by the two equations
					\[ u = 0, \quad w = 0. \]
					Pulling back these equations along $j_{\dR}$, we obtain the following:

					\begin{prop}
						\label{thm: equations |,|}
						The set $\mathcal X(\ZZ_p)_{\{2,q\},2}^{|,|}$ is cut out in $\mathcal X(\ZZ_p)$ by the two equations
						\begin{equation}
							\label{subspace equation |,|}
							\log(z) = 0, \quad \Li_2(z) = 0.
						\end{equation}
					\end{prop}
					
					As in the case of one prime above, the vanishing locus $\mathcal X(\ZZ_p)_{\{2,q\},2}^{|,|}$ consists of the nontrivial $(p-1)$-st roots of unity $\zeta$ for which $\Li_2(\zeta) = 0$. This includes in particular $-1$, which is a solution of the $S$-unit equation since $2 \in S$. Note that the set $\mathcal X(\ZZ_p)_{\{2,q\},2}^{|,|}$ does not depend on the specific prime~$q$ and in particular not on the coefficients $a_{2,q}$ and $a_{q,2}$.
					
					We turn to the points of type $(|,-)$. The restriction of $\loc_p$ to the subspace $\Sel_{S,2}^{|,-}$ is given by
					\[ \loc_p(0, x_q, y_2, 0) = \begin{pmatrix}
						\log(q) x_q\\
						\log(2) y_2\\
						a_{q,2} x_q y_2
					\end{pmatrix}.\]
					In the coordinates $u,v,w$ on $\AA^3$, the image of $\Sel_{S,2}^{|,-}$ is therefore cut out by the equation
					\[ a_{q,2} uv - \log(2) \log(q) w = 0. \]
					Pulling back along $j_{\dR}$ gives the following equation for $\mathcal X(\ZZ_p)_{S,2}^{|,-}$:
					\begin{equation}
						\label{eq: subspace equation |,-, first form}
						a_{q,2} \log(z) \log(1-z) + \log(2)\log(q) \Li_2(z) = 0.
					\end{equation}
					Using the twisted anti-symmetry relation $\log(2)\log(q) = a_{2,q} + a_{q,2}$ and the functional equation $\Li_2(z) + \Li_2(1-z) = -\log(z) \log(1-z)$, this can be written in a more symmetric form as follows:
					
					\begin{prop}
						\label{thm: equations |,-}
						The set $\mathcal X(\ZZ_p)_{\{2,q\},2}^{|,-}$ is cut out in $\mathcal X(\ZZ_p)$ by the equation 
						\begin{equation}
							\label{subspace equation |,-}
							a_{2,q} \Li_2(z) = a_{q,2} \Li_2(1-z).
						\end{equation}
					\end{prop}
					
					Note that $a_{2,q}$ and $a_{q,2}$ are not both zero since their sum is $\log(2) \log(q) \neq 0$ by the twisted anti-symmetry relation. Equation~\eqref{subspace equation |,-} is thus not trivial.
					
					\begin{rem}
						We note for later use that $\mathcal X(\ZZ_p)_{S,2}^{-,|}$ is defined by the equation
						\begin{equation}
							\label{subspace equation -,|}
							a_{2,q} \Li_2(1-z) = a_{q,2} \Li_2(z).
						\end{equation}
						This is obtained from the equation for $\mathcal X(\ZZ_p)_{S,2}^{|,-}$ by using the symmetry $z\mapsto1-z$.
					\end{rem}
					
					\begin{rem}
						\label{rem: solutions 2 and -1}
						Since $\Li_2(-1) = \Li_2(2) = 0$, Equations~\eqref{subspace equation |,-} and~\eqref{subspace equation -,|} are both satisfied for $z = -1$ and $z = 2$, hence both elements are contained in both loci $\mathcal{X}(\ZZ_p)_{S,2}^{|,-}$ and $\mathcal{X}(\ZZ_p)_{S,2}^{-,|}$. While $-1$ and $2$ are indeed $S$-integral points (since $2 \in S$), considering the valuations of~$z$ and $1-z$, we expect a priori only that $-1 \in \mathcal{X}(\ZZ_p)_{S,2}^{|,-}$ and $2 \in \mathcal{X}(\ZZ_p)_{S,2}^{-,|}$.
					\end{rem}
					
					
					Observe that the equations $\log(z) = \Li_2(z) = 0$ for $\mathcal X(\ZZ_p)_{S,2}^{|,|}$ imply the equation \eqref{eq: subspace equation |,-, first form} for $\mathcal X(\ZZ_p)_{S,2}^{|,-}$, so we have the inclusion
					\begin{equation}
						\label{eq: refined CK set inclusion}
						\mathcal X(\ZZ_p)_{S,2}^{|,|} \subseteq \mathcal X(\ZZ_p)_{S,2}^{|,-}.
					\end{equation}
					
					As a consequence of the 2-transitivity and of the inclusion~\eqref{eq: refined CK set inclusion}, the complete set $\mathcal X(\ZZ_p)_{S,2}^{\min}$ can be computed from $\mathcal X(\ZZ_p)_{S,2}^{|,-}$ by taking $S_3$-orbits.
					
					The results are summarized in the following theorem.
					
					\begin{theorem}[= Theorem~\ref{thm:main_s=2}]
						\label{refined Chabauty-Kim equations}
						Let $S = \{2, q\}$ for some odd prime~$q$, and let $p$ be a prime not in~$S$. The refined Chabauty--Kim set $\mathcal X(\ZZ_p)_{S,2}^{\min}$, up to $S_3$-orbits, is cut out in $\mathcal X(\ZZ_p)$ by the equation
						\[ a_{2,q} \Li_2(z) = a_{q,2} \Li_2(1-z). \]
					\end{theorem}


					\subsection{Power series}
					Now that we have found explicit equations for the refined Chabauty--Kim loci, we want to use these equations to determine the sets $\mathcal{X}(\ZZ_p)_{S,2}^\min$ for $S= \{2,q\}$. In order to bound their size, we can compute power series for our defining equations on residue discs and analyze their Newton polygon. We carry this out in the case~$p=3$, where we have $\mathcal{X}(\ZZ_3) = 2 + 3\ZZ_3$, so that only a single residue disc needs to be considered.
					
					We start by determining power series expansions of the components of the de Rham Kummer map
					\begin{align*}
						j_{\dR} \colon \mathcal{X}(\ZZ_p) &\to U_2^{\dR} \\
						z &\mapsto (\log(z), \log(1-z), -\Li_2(z)).
					\end{align*}
					in the residue disc~$]2[(\ZZ_p)$ for an arbitrary odd prime~$p$.

					
					
					Recall that the functions $\log(z)$ and $\Li_2(z)$ are defined as
					\footnote{It is more customary to define $\log(z)$ as a Coleman integral from~$1$ to~$z$, rather than from $\overrightarrow{01}$. However, it makes no difference which start point we use to define $\log$, since $\int_{\overrightarrow{01}}^1\frac{\diff x}{x}=0$. To see this, note that the integral can be computed on~$\mathbb G_m$, which is canonically isomorphic to the punctured tangent space at the origin of $\AA^1$. This yields a canonical isomorphism between the fiber functors at the point $1$ and the tangent vector $\overrightarrow{01}$, which implies the vanishing $\int_{\overrightarrow{01}}^1\frac{\diff x}{x}=0$.}
					the (iterated) Coleman integrals
					\begin{align*}
						\log(z) &= \int_{\overrightarrow{01}}^z\frac{\diff x}{x} \\
						\Li_2(z) &= \int_{\overrightarrow{01}}^z\frac{\diff x}{x}\frac{\diff x}{1-x} \,.
					\end{align*}
					
					Using additivity of abelian Coleman integrals, we thus find that for $z\in 2+p\ZZ_p$, the logarithm $\log(z)$ is given by
					\[
					\log(z) = \underbrace{\int_{\overrightarrow{01}}^2 \frac{\diff x}{x}}_{=\log(2)} + \underbrace{\int_{2}^{z} \frac{\diff x}{x}}_{\text{tiny integral}}
					= \log(2) - \sum_{k=1}^\infty \frac{(2-z)^k}{k2^k}.
					\]
					
					Similarly, for $z\in2+p\ZZ_p$ (so $1-z\in-1+p\ZZ_p$) the logarithm of~$1-z$ is given by
					\begin{equation}
						\label{power series log(1-z)}
						\Li_1(z) = \log(1-z) = \underbrace{\int_{\overset{\to}{01}}^{-1} \frac{\diff x}{x}}_{=\log(-1)=0} + \underbrace{\int_{-1}^{1-z} \frac{\diff x}{x}}_{\text{tiny integral}}
						= -\sum_{k=1}^\infty \frac{(2-z)^k}{k}.
					\end{equation}
					
					For $\Li_2(z)$, using the path composition rule for iterated Coleman integrals, we get 
					\begin{equation}
						\label{Li2 additivity in endpoints}
						\Li_2(z) = \Li_2(2) + \left(\int_2^z\frac{\diff x}{x}\right)\cdot\Li_1(2) + \int_{2}^z \frac{\diff x}{x}\frac{\diff x}{1-x} \,.
					\end{equation}
					
					Since $\Li_2(2)=0$ (from $\Li_2(z) + \Li_2(1-z) = - \log(z)\log(1-z)$ and $\Li_2(z) + \Li_2(z^{-1}) = -\frac12(\log(z))^2$), and $\Li_1(2) = 0$ (from $\Li_1(z) = - \log(1-z)$), we have
					\[\Li_2(z) = -\int_{t= 0}^{z-2} \frac{\diff t}{t + 2} \frac{\diff t}{t+1} \text.\] 
					This can be calculated to be
					\begin{equation}
						\label{power series Li_2(z)}
						\Li_2(z) = -\sum_{k > i \geq 1} \frac{1}{k}{\frac{1}{i 2^{k-i}}} (2-z)^k.
					\end{equation}
					In concrete terms,
					\begin{dmath}
						\Li_2(z)= -\frac14 (2-z)^2 - \frac16 (2-z)^3 - \frac5{48} (2-z)^4 - \frac1{15}(2-z)^5 - \frac2{45} (2-z)^6 + O((2-z)^7),
					\end{dmath}
					where $O((2-z)^7)$ represents all terms of the form $\alpha(2-z)^k$ for $k \geq 7$.
					
					It will be useful to calculate also the power series expansion of $\Li_2(1-z)$ on the residue disk $2 + p\ZZ_p$. Then $1-z$ lies in the same residue disk as $-1$, so by using the path composition rule similarly to~\eqref{Li2 additivity in endpoints} we obtain
					\[ \Li_2(1-z) = \Li_2(-1) + \left(\int_{t=0}^{2-z}  \frac{\diff t}{t-1}\right)\cdot \Li_1(-1) -\int_{t=0}^{2-z} \frac{\diff t}{1-t} \frac{\diff t}{2-t}. \]
					The first summand is $\Li_2(-1) = 0$. In the second summand, we have
					\[ \Li_1(-1) = -\log(1-(-1)) = -\log(2), \]
					and the integral equals
					\[ \int_{t=0}^{2-z} \frac{\diff t}{t-1} = \log(1-z) - \log(-1) = \log(1-z), \]
					for which we have already calculated the power series in~\eqref{power series log(1-z)} above. For the final summand we calculate the tiny iterated integral as
					\begin{align*}
						\int_{t=0}^{2-z} \frac{\diff t}{1-t} \frac{\diff t}{2-t}
						&= \int_{t=0}^{2-z} \frac{\diff t}{1-t} \sum_{i=1}^\infty \frac{t^i}{i 2^i}
						= \int_{t=0}^{2-z} \sum_{i=1}^\infty \sum_{j=1}^\infty \frac{t^{i+j-1} \diff t}{i 2^i} \\
						&= \sum_{i=1}^\infty \sum_{j=1}^\infty \frac{(2-z)^{i+j}}{(i+j) i 2^i} 
						= \sum_{k=2}^\infty \sum_{i=1}^{k-1} \frac{(2-z)^k}{k i 2^i}.
					\end{align*}
					Combining the three summands, we obtain the power series
					\begin{equation}
						\label{power series Li2(1-z)}
						\Li_2(1-z) = \sum_{k=1}^\infty \frac1{k} \Bigl( \log(2) - \sum_{i=1}^{k-1} \frac{1}{i 2^i} \Bigr) (2-z)^k.
					\end{equation}
					
					\subsection{Newton polygon analysis}
					Let now $S = \{2,q\}$ with $\geq 5$ an odd prime. We choose $p = 3$ in the Chabauty--Kim method. We will study the Newton polygon for one of the equations~\eqref{subspace equation |,-} or~\eqref{subspace equation -,|}:
					\begin{align*}
						a_{2,q} \Li_2(z) &= a_{q,2} \Li_2(1-z), \\
						a_{q,2} \Li_2(z) &= a_{2,q} \Li_2(1-z).
					\end{align*}
					The first cuts out the refined Chabauty--Kim locus~$\mathcal{X}(\ZZ_3)_{S,2}^{|,-}$, the second cuts out $\mathcal{X}(\ZZ_3)_{S,2}^{-,|}$. The two loci are transformed into each other by $z \mapsto 1-z$, hence either of them can be used to generate the full refined Chabauty--Kim locus~$\mathcal{X}(\ZZ_3)_{S,2}^\min$ via taking $S_3$-orbits. We shall use the first equation if $v_3(a_{2,q}) > v_3(a_{q,2})$ and the second otherwise. The equation under consideration is thus
					\begin{align*}
						A \Li_2(z) &= a \Li_2(1-z)
					\end{align*}
					with $(A,a)$ equal to $(a_{2,q}, a_{q,2})$ or $(a_{q,2}, a_{2,q})$ in such a way that $v_3(A) \geq v_3(a)$. 
					
						
						Using the power series expansions~\eqref{power series Li_2(z)} and~\eqref{power series Li2(1-z)} for $\Li_2(z)$ and $\Li_2(1-z)$, respectively, we obtain the power series
						\[ f(z) \coloneqq \sum_{k=1}^\infty c_k (2-z)^k \]
						with coefficients
						\begin{align}
							\label{general coefficient formula}
							c_k &= \frac{1}{k} \sum_{i=1}^{k-1} \frac1{(k-i)2^i} A + \frac{1}{k} \Bigl( \log(2) - \sum_{i=1}^{k-1} \frac1{i 2^i} \Bigr) a,
						\end{align}
						which converges on $\mathcal X(\ZZ_3) = 2 + 3 \ZZ_3$ and defines the set $\mathcal X(\ZZ_3)_{S,2}^{|,-}$ or $\mathcal X(\ZZ_3)_{S,2}^{-,|}$. For instance, the first four coefficients are given by
						\begin{align*}
							c_1 &= \log(2) a,\\
							c_2 &= \tfrac14 A + \tfrac12 \Bigl( \log(2) - \tfrac12 \Bigr) a,\\
							c_3 &= \tfrac16 A + \tfrac13 \Bigl( \log(2) - \tfrac58 \Bigr) a,\\
							c_4 &= \tfrac{5}{48} A + \tfrac14 \Bigl( \log(2) - \tfrac23 \Bigr) a.
						\end{align*}
						We study the Newton polygon of $f(z)$ to count its roots in the disk $2+3\ZZ_3$.
						Details on Newton polygon analysis for $p$-adic power series can be found in \cite[Chapter IV, \S 4]{koblitz:padic}. In determining this Newton polygon, we will use the following determination of the valuations of $3$-adic logarithms.

						
						
						
						\begin{lemma}
							\label{log and dilog valuation}
							For $z \in 2+ 3\ZZ_3$, the 
							valuation of $\log(z)$ is given by
							\[ v_3(\log(z)) = v_3(z+1). \]
							\end{lemma}
							
							\begin{proof}
								In the series expansion
								\[ \log(z) = - \sum_{k=1}^\infty \frac1{k} (z+1)^k, \]
								the first summand dominates, i.e.\ for all $k \geq 2$ we have
								\[ v_3\Bigl(\frac1{k} (z+1)^k\Bigr) - v_3(z+1) = (k-1) v_3(z+1) - v_3(k) \geq (k-1) - \log_3(k) > 0, \]
								which shows the claim. 
							\end{proof}

							\begin{prop}
								\label{newton polygon analysis result}
								Let $S = \{2, q\}$ for~$q>3$ prime, and let $p = 3$. Then the refined Chabauty--Kim set $\mathcal X(\ZZ_3)_{S,2}^{\min}$ contains $\{2,-1,\frac12\}$ and at most one more $S_3$-orbit of points. The second orbit is present if and only if 
								\begin{equation}\tag{$\dag$}
									\min\{ v_3(a_{2,q}), v_3(a_{q,2}) \} = 1 + v_3(\log(q)).
								\end{equation}
							\end{prop}
							\begin{proof}
								As above, let $\{A, a\}$ be $\{a_{2,q}, a_{q,2}\}$, assigned such that $v_3(A) \geq v_3(a)$. Returning to the study of the function $f(z) = A \Li_2(z) - a \Li_2(1-z)$ defining $\mathcal X(\ZZ_3)_{S,2}^{|,-}$ or $\mathcal X(\ZZ_3)_{S,2}^{-,|}$, write $z = 2-3t$, then the coefficients of $f(t)$ as a power series in $t$ are $3^k c_k$. We analyze the Newton polygon of this power series. The $3$-adic valuation of the $k$-th coefficient is given by $k + v_3(c_k)$. For $k = 1$, this is
								\[ 1+v_3(c_1) = 1 + v_3(\log(2) a) = 2 + v_3(a) \]
								by Lemma~\ref{log and dilog valuation}. 
								For $k \geq 2$, the difference of valuations between the first and the $k$-th coefficient is
								\begin{align*}
									(k + v_3(c_k)) - (1 + v_3(c_1)) &= k - 2 - v_3(a) + v_3(c_k) \\
									&\geq k - 2 - v_3(a) - v_3(k) - \max_{1\leq i\leq k-1} v_3(i) + v_3(a) \\
									&= k - 2 - v_3(k) - \max_{1\leq i\leq k-1} v_3(i).
								\end{align*}
								The last expression satisfies 
								\[ k - 2 - v_3(k) - \max_{1\leq i\leq k-1} v_3(i) \begin{cases}
									= 0, & \text{for $k = 2,3$},\\
									> 0, & \text{for $k \geq 4$},
								\end{cases}\]
								as one checks by hand for $k = 4$ and for higher $k$ via the estimate
								\[ k - 2 - v_3(k) - \max_{1\leq i\leq k-1} v_3(i) \geq k - 2 - \log_3(k) - \log_3(k-1). \]
								Let $\nu \coloneqq 2 + v_3(a)$, then the Newton polygon of $f(t)$ has the form
								\[ (0, \infty),\; (1, \nu),\; (2, \geq \nu),\; (3, \geq \nu),\; (4, > \nu), \ldots\]
								By Remark~\ref{rem: solutions 2 and -1}, the elements $z = 2$ and $z = -1$ are two known solutions to the equation $f(z) = 0$. The first line segment of slope $-\infty$ belongs to the root $t = 0$ of~$f(t)$, corresponding to $z = 2$. Corresponding to $z = -1$ we have $t = 1$ as a second known root, so that there is a segment of slope $0$. Hence, the first $\geq$ is actually an equality and the point $(2,\nu)$ belongs to the Newton polygon. There is at most one other root in $\mathcal{O}_{\mathbb{C}_p}$ before the Newton polygon continues with positive slopes, so there is at most one additional $S_3$-orbit of points in $\mathcal{X}(\ZZ_3)_{S,2}^{\min}$ besides $\{2,-1,\frac12\}$.
								
								The extra root of $f(t)$, if it is present, is a priori an element of $\mathcal{O}_{\mathbb{C}_p}$ but in fact it must necessarily be contained in $\ZZ_3$: it is algebraic over~$\QQ_p$ by the Weierstrass preparation theorem, and if it were not contained in $\ZZ_p$ then taking Galois conjugates would produce even more roots in $\mathcal{O}_{\mathbb{C}_p}$ contradicting the shape of the Newton polygon. We conclude that an extra root of $f(t)$ corresponds precisely to a second $S_3$-orbit of points in $\mathcal{X}(\ZZ_3)_{S,2}^{\min}$.
								
								It remains to prove the criterion for when this happens. The line segment in question has nonnegative slope if and only if $3 + v_3(c_3) = \nu$, which is equivalent to $v_3(c_3) = v_3(a) - 1$. Using the twisted antisymmetry relation $A + a = \log(2) \log(q)$, we have
								\begin{align*}
									c_3 &=  \tfrac16 A + \tfrac13 \bigl( \log(2) - \tfrac58 \bigr) a\\
									&= \tfrac16 \log(2) \log(q) + \tfrac 13 \bigl( \underbrace{\log(2) - \tfrac98}_{\text{valuation $1$}} \bigr) a. 
								\end{align*} 
								The twisted anti-symmetry relation also implies
								\begin{equation}
									\label{eq: twisted ineq}
									1 + v_3(\log(q)) \geq \min\{v_3(a_{2,q}), v_3(a_{q,2})\} = a.
									\tag{$\ast$}
								\end{equation}
								Hence, the second summand in the formula for $c_3$ has valuation $v_3(a)$, the first summand has valuation $\geq v_3(a) - 1$. It follows that we have $v_3(c_3) = v_3(a) - 1$ if and only if the inequality~\eqref{eq: twisted ineq} is an equality.
							\end{proof}
							
							\begin{cor}
								Let~$q\geq5$ be a Fermat or Mersenne prime. Then the refined version of Kim's Conjecture holds for~$\mathcal{X}=\PP^1\setminus\{0,1,\infty\}$, $S=\{2,q\}$, $n=2$ and~$p=3$. That is, we have
								\[
								\mathcal{X}(\ZZ_3)^{\min}_{\{2,q\},2} = \mathcal{X}(\ZZ[\tfrac1{2q}]) .
								\]
								\begin{proof}
									$\mathcal{X}(\ZZ_3)_{S,2}^\min$ consists of $\{2,-1,\frac12\}$ and at most one other $S_3$-orbit of points. But on the other hand it contains $\mathcal{X}(\ZZ[\tfrac1{2q}])$, which consists of $\{2,-1,\frac12\}$ and the $S_3$-orbit of~$q$ (if~$q$ is Fermat) or~$-q$ (if~$q$ is Mersenne). So we have equality.
								\end{proof}
							\end{cor}
							
							\begin{rem}
								One can also numerically check the criterion in Proposition~\ref{newton polygon analysis result} to find primes~$q$ for which the refined Kim's Conjecture holds in depth~$2$. Using our SAGE code~\cite{dcw_coefficients} to compute the $3$-adic coefficients~$a_{2,q}$ and $a_{q,2}$, we have checked the criterion for all~$5\leq q\leq 1000$, finding that $\min\{ v_3(a_{2,q}), v_3(a_{q,2}) \} \neq 1 + v_3(\log(q))$ for 31 values of~$q$, namely 
								\begin{align*}
									q =\; &19,\; 37,\; 53,\; 107,\; 109,\; 163,\; 181,\; 199,\; 269,\; 271,\; 379,\; \\
									&431,\; 433,\; 487,\; 523,\; 541,\; 577,\; 593,\; 631,\; 701,\; 739,\; \\
									&757,\; 809,\; 811,\; 829,\; 863,\; 883,\; 919,\; 937,\; 971,\; 991.
								\end{align*}
								For these values of~$q$, the refined Kim's Conjecture holds as well:
								\[\mathcal{X}(\ZZ_3)^{\min}_{\{2,q\},2} = \mathcal{X}(\ZZ[\tfrac1{2q}]) = \{-1,\tfrac12, 2\}.\]
							\end{rem}

							\subsection{The case $q=3$}
							\label{case3}
							
							Now consider the set $S = \{2,3\}$. As discussed in Section~\ref{sec: classification of solutions}, the prime $q = 3$ is special in that there are solutions to the $\{2,3\}$-unit equation of all three kinds: Fermat ($3-2 = 1$), Mersenne ($-3+4 = 1$) and Catalan ($9-8 = 1$). Together with the $\{2\}$-integral solution ($-1+2=1$), this gives four $S_3$-orbits of $\mathcal{X}(\ZZ[1/6])$, forming $21$ solutions in total:
							\begin{align*}
								\{-1, 1/2, 2\} &\cup \{-2, -1/2, 1/3, 2/3, 3/2, 3\} \cup \{-3, -1/3, 1/4, 3/4, 4/3, 4\}\\ &\cup \{-8, -1/8, 1/9, 8/9, 9/8, 9\}.
							\end{align*}
							
							Let $p \neq 2,3$ be a prime. The three $S$-integral points $3$, $-3$, $9$ lead to three different formulas for the coefficients $a_{2,3}$ and $a_{3,2}$. Viewing $3 = 2^1+1$ as a Fermat prime yields
							\begin{align}
								\label{a23 fermat}
								a_{2,3} &= -\Li_2(-2), \quad a_{3,2} = -\Li_2(3)
							\end{align}
							by Lemma~\ref{lem: Fermat Mersenne DCW coefficients}; viewing $3 = 2^2-1$ as a Mersenne prime yields
							\begin{align}
								\label{a23 mersenne}
								a_{2,3} &= -\frac12 \Li_2(4), \quad a_{3,2} = -\frac12\Li_2(-3);
							\end{align}
							and using the commutativity of the Chabauty--Kim diagram for the Catalan solutions $z = -8$ and $z = 9$, which satisfy $j_S(-8) = (3,0,0,2)$ and $j_S(9) = (0,2,3,0)$, yields
							\begin{align}
								\label{a23 catalan}
								a_{2,3} &= -\frac16 \Li_2(-8), \quad a_{3,2} = -\frac16 \Li_2(9).
							\end{align}
							
							Thus, the Chabauty--Kim diagram yields as a byproduct the following identities of dilogarithms:
							\begin{align*}
								\Li_2(-2) &= \frac12 \Li_2(4) = \frac16 \Li_2(-8), \\
								\Li_2(3) &= \frac12\Li_2(-3) = \frac16 \Li_2(9).
							\end{align*}
							
							According to Theorem~\ref{refined Chabauty-Kim equations}, the refined Chabauty--Kim set $\mathcal X(\ZZ_p)_{\{2,3\},2}^{|,-}$ is cut out by the equation
							\begin{equation}
								\label{equation q = 3}
								a_{2,3} \Li_2(z) = a_{3,2} \Li_2(1-z),
							\end{equation}
							with the coefficients $a_{2,3}$ and $a_{3,2}$ given by the equivalent equations~\eqref{a23 fermat}, \eqref{a23 mersenne}, \eqref{a23 catalan}. The points $-1$, $3$, $-3$, $9$ are all of type $(|,-)$ and are therefore solutions of~\eqref{equation q = 3}. 
							
							Let us now choose $p = 5$, the smallest possible choice since $p$ has to be different from $2$ and $3$. According to Theorem~\ref{refined Chabauty-Kim equations}, the refined Chabauty--Kim set $\mathcal X(\ZZ_5)_{\{2,3\},2}^{|,|}$ is cut out by the equations $\log(z) = \Li_2(z) = 0$. The simultaneous roots when $p=5$ are the 4th roots of unity $\zeta$ satisfying $\Li_2(\zeta) = 0$. Numerical computation shows that this is not the case for $\pm i$ and the fact that $\Li_2(-1) = 0$ shows that the only root in $\mathcal X(\ZZ_5)$ is $z= -1$. However, computer calculations show that for the set $\mathcal X(\ZZ_5)_{\{2,3\},2}^{|,-}$ there is, in addition to the known solutions above, one extra solution of~\eqref{equation q = 3} in the residue disk $3+5\ZZ_5$ which does not correspond to a solution of the $S$-unit equation and appears transcendental. Recently, one of the authors (M.L.) showed that this extra solution can be eliminated by going to depth~$4$, thus confirming the refined Kim's Conjecture for $S=\{2,3\}$, $p=5$ and $n=4$.

							\begingroup
							\sloppy
							\printbibliography
							\endgroup

						\end{document}